\documentclass[12pt]{amsart}
\usepackage{amsfonts,amsmath,amssymb,amsthm}
\usepackage{mathrsfs}
\usepackage{graphics}
\usepackage{fullpage}
\usepackage{hyperref}
\usepackage[english]{babel}
\usepackage[latin1]{inputenc}
\usepackage{cite}
\allowdisplaybreaks
\numberwithin{equation}{section}
\usepackage[dvipsnames]{xcolor}

\newtheorem{theorem}{Theorem}[section]
\newtheorem{lemma}[theorem]{\bf Lemma}
\newtheorem{corollary}[theorem]{Corollary}
\newtheorem{prop}[theorem]{Proposition}
\newtheorem{definition}[theorem]{Definition}

\newtheorem{remark}[theorem]{\textbf{Remark}}

\newcommand{\N}{\mathbb{N}}

\newcommand{\R}{\mathbb{R}}

\newcommand{\calS}{\mathcal{S}}

\DeclareMathOperator{\dist}{dist}
\DeclareMathOperator*{\esssup}{ess \, sup}
\DeclareMathOperator*{\essinf}{ess \, inf}

\newcommand{\vf}{\mathbf{f}}

\newcommand{\ve}{\mathbf{e}}

\newcommand{\calA}{\mathcal{A}}

\newcommand{\calW}{\mathcal{W}}

\newcommand{\tA}{\tilde{A}}

\newcommand{\bs}{\backslash}
\newcommand{\bk}{\backslash}

\DeclareMathOperator{\op}{op}
\newcommand{\loc}{\text{loc}}

\def\Xint#1{\mathchoice
   {\XXint\displaystyle\textstyle{#1}}%
   {\XXint\textstyle\scriptstyle{#1}}%
   {\XXint\scriptstyle\scriptscriptstyle{#1}}%
   {\XXint\scriptscriptstyle\scriptscriptstyle{#1}}%
   \!\int}
\def\XXint#1#2#3{{\setbox0=\hbox{$#1{#2#3}{\int}$}
     \vcenter{\hbox{$#2#3$}}\kern-.5\wd0}}

\def\avgint{\Xint-}
\def\dashint{\Xint-}

\newcommand{\pp}{{p(\cdot)}}
\newcommand{\Pp}{\mathcal{P}}
\newcommand{\Lpp}{L^\pp}
\newcommand{\cpp}{{p'(\cdot)}}
\newcommand{\Lcpp}{L^\cpp}
\newcommand{\rp}{{r(\cdot)}}

\newcommand{\Lrp}{L^\rp}

\newcommand{\qp}{{q(\cdot)}}
\newcommand{\Lqp}{L^{\qp}}
\newcommand{\cqp}{{q'(\cdot)}}
\newcommand{\Lcqp}{L^{\cqp}}

\newcommand{\up}{{u(\cdot)}}

\newcommand{\Lup}{L^\up}
\newcommand{\vp}{{v(\cdot)}}

\newcommand{\Lvp}{L^\vp}
\newcommand{\bbW}{{\mathbb{W}}}

\begin{document}

\title[The reverse H\"older inequality for $\calA_\pp$ weights]
{The reverse H\"older inequality for $\calA_\pp$ weights\\ with applications to matrix weights}

\author{David Cruz-Uribe OFS}
\address{David Cruz-Uribe, OFS \\Department of Mathematics, University of Alabama, Tuscaloosa, AL 35487, USA}
\email{dcruzuribe@ua.edu}

\author{Michael Penrod}
\address{Michael Penrod \\ Department of Mathematics, University of Alabama, Tuscaloosa, AL 35487, USA}
\email{mjpenrod@crimson.ua.edu}

\thanks{The first author is partially supported by a Simons Foundation Travel Support for Mathematicians Grant and by NSF grant DMS-2349550. The authors would like to thank the anonymous referees for their careful reading of our paper and for the suggestions which improved our exposition. }



\date{\today}

\begin{abstract}
In this paper we prove a reverse H\"{o}lder inequality for the variable exponent  Muckenhoupt weights $\calA_\pp$, introduced in~\cite{MR2927495}. All of our estimates are quantitative, showing the dependence of the exponent function on the $\calA_\pp$ characteristic.  As an application, we use  the reverse H\"{o}lder inequality to prove that the  matrix $\calA_\pp$ weights, introduced in~\cite{ConvOpsOnVLS}, have both a  right and left-openness property.  This result is new even in the scalar case.
\end{abstract}
\maketitle
\section{Introduction}

In this paper, we  develop the structure theory of weights in the variable exponent setting by proving a version of the reverse H\"older inequality for the class of $\calA_\pp$ weights.  Before stating our main theorem, we sketch earlier results to provide some context.  The study of $A_p$ weights dates back to the early 1970s when Muckenhoupt \cite{MR0293384} proved the Hardy-Littlewood maximal operator is bounded on $L^p(v)$ if and only if $v \in A_p$. 
Given $1< p < \infty$, a weight $v$ is a (scalar) $A_p$ weight if 
\[[v]_{A_p} : = \sup_Q \dashint_Q v(x)\, dx \left( \dashint_Q v(y)^{1-p'}\, dy \right)^{p-1}<\infty,\]
where the supremum is taken over all cubes $Q\subset \R^n$.
A rich structural theory for these weights developed quickly (see \cite{duoandikoetxea_fourier_2000, MR0736522,MR0807149}). The fine properties of $A_p$ weights play an important role in many of their applications. Of particular interest is the reverse H\"older inequality, introduced by  Coifman and Fefferman \cite{coifman_weighted_1974}.  They used it to give a simpler proof that the Hardy-Littlewood maximal operator is bounded, and to prove that  singular integral operators are bounded on $L^p(v)$ for $v\in A_p$. They showed that given a weight $v\in A_p$, there exist constants $C$ and $r>1$ such that for all cubes $Q$, 
\[ \dashint_Q v(x)^r\, dx \leq C \left( \dashint_Q v(x)\, dx \right)^r.\]
This is called a reverse H\"older inequality since the opposite inequality is just a normalized version of H\"older's inequality.
By H\"older's inequality, it follows at once from the definition  that  the $A_p$ classes are nested:  $A_p \subset A_q$ for all $q >p$. Moreover, as a consequence of the reverse H\"older inequality,  they are left-open: if  $v\in A_p$, then $v \in A_{p-\epsilon}$ for some $\epsilon>0$. 

To generalize  $A_p$ weights to the variable Lebesgue spaces we need to define an alternative version of $A_p$, which we denote by $\calA_p$. Given $1 < p <\infty$, a weight $w$ is a scalar $\calA_p$ weight if 
\[ [w]_{\calA_p} : = \sup_Q \left( \dashint_Q w(x)^p\, dx \right)^{\frac{1}{p}} \left( \dashint_Q w(y)^{-p'}\, dy\right)^{\frac{1}{p'}}<\infty.\]
Notice that $w \in \calA_p$ if and only if $v=w^p \in A_p$ with $[w]_{\calA_p} =[v]_{A_p}^{\frac{1}{p}}$.  The definition of classical $A_p$ weights is based on viewing the weight $v$ as a measure in the $L^p(v)$ norm, i.e., defining
\[ \|f\|_{L^p(v)} = \left( \int_{\R^n} |f(x)|^p\, v(x) dx \right)^{\frac{1}{p}}.\]
On the other hand, the definition of $\calA_p$ is based on viewing the weight $w=v^{\frac{1}{p}}$ as a multiplier, i.e., we define the $L^p(w)$ norm by 
\[\|f\|_{L^p(w)}= \|wf\|_{L^p(\R^n)} = \left( \int_{\R^n} |w(x) f(x)|^p\, dx \right)^{\frac{1}{p}}.\]
 This approach of using weights as multipliers was first adopted by Muckenhoupt and Wheeden~\cite{MR340523} to define the ``off-diagonal" $A_{p,q}$ weights used with fractional integral operators.  One consequence of this approach is that the ``duality" of the weights has a much more symmetric expression:  with this definition, $w\in \calA_p$ if and only if $w^{-1}\in \calA_{p'}$, whereas in the classical case, $v\in A_p$ if and only if $v^{1-p'}=v^{-p'/p}\in A_{p'}$. 

The $\calA_p$ weights also satisfy a reverse H\"older inequality:  this follows at once from the classical reverse H\"older applied to the weight $v=w^p$:
\begin{equation} \label{eqn:calAp-holder}
\left( \dashint_Q w(x)^{rp}\, dx \right)^\frac{1}{rp} \leq C_p \left( \dashint_Q w(x)^p\, dx \right)^\frac{1}{p}.
\end{equation}
Since $w \in \calA_p$ and $w^{-1}\in \calA_{p'}$ we can use this inequality to show that the class $\calA_p$ is both  right and left open:  there exists $\epsilon>0$ such that $w\in \calA_{p-\epsilon}$ and $w\in \calA_{p+\epsilon}$.  However, unlike the classical $A_p$ weights, the $\calA_p$ weights are not nested.  For example, on the real line, if we let $w(x)=|x|^{-\frac{1}{2}}$, then $w\in \calA_p$ for $1<p<2$, but not for $p\geq 2$.  

\medskip

The definition of $\calA_p$ weights leads to a natural generalization in the variable exponent setting.  We can rewrite the definition of $[w]_{\calA_p}$  using $L^p$ norms, i.e.,
\[[w]_{\calA_p} = \sup_Q |Q|^{-1} \|w\chi_Q\|_{L^p(\R^n)} \|w^{-1}\chi_Q\|_{L^{p'}(\R^n)},\]
and then replace the $L^p$ norms with variable exponent $L^\pp$ norms. Given an exponent function $\pp$, a weight $w$ is an $\calA_\pp$ weight if 
\[ [w]_{\calA_\pp} : = \sup_Q |Q|^{-1} \| w\chi_Q\|_{\Lpp(\R^n)} \| w^{-1}\chi_Q\|_{\Lcpp(\R^n)} <\infty.\]
(See Section \ref{sec:Prelim} for the definitions of exponent functions and variable Lebesgue spaces.) 
In \cite{MR2927495}, the first author, Fiorenza, and Neugebauer proved the Hardy-Littlewood maximal operator satisfies strong and weak-type weighted norm inequalities if and only if $w \in \calA_\pp$. (See also~\cite{MR4387458, MR2837636}.)   Much less is known about the fine properties of $\calA_\pp$ weights than in the constant exponent setting, though   the first author and Wang~\cite{MR3572271}  proved that Rubio de Francia extrapolation holds for $\calA_\pp$ weights. 

In this paper we begin the study of the fine properties of $\calA_\pp$ weights.  In our main result, we prove that they satisfy a reverse H\"{o}lder inequality defined in terms of variable Lebesgue space norms.  For brevity, here we omit some technical definitions:  see Section~\ref{sec:Prelim} for precise details.  
\begin{theorem}\label{thm:NormRH}
Let $\pp \in \Pp(\R^n) \cap LH(\R^n)$ with $p_+ <\infty$ and let $w$ be a scalar $\calA_\pp$ weight. Then there exists a constant $C_{\pp}$ and an exponent $r>1$ such that for all cubes $Q\subset \R^n$, 
\begin{align}\label{ineq:NormRH}
 |Q|^{-\frac{1}{rp_Q}} \| w\chi_Q\|_{L^{r\pp}(\R^n)} \leq C_{\pp} |Q|^{-\frac{1}{p_Q}} \| w\chi_Q\|_{\Lpp(\R^n)};
 \end{align}
in particular, we may take 
\[ C_\pp = C^*[w]_{\calA_\pp}^{\frac{C_\infty p_+}{p_-^2}(p_++1)},\]
where $C^* = C(n,\pp,C_\infty, C_D ,C_0, D_1, D_2)$, and
\[r=1 + \frac{1}{C_*[w]_{\calA_\pp}^{\left(1+2\frac{C_\infty p_+}{p_\infty p_-}\right)p_+}},\]
where $C_* = C(n,\pp,C_\infty,C_D)$.
\end{theorem}

The statement of Theorem~\ref{thm:NormRH} is analogous to the quantitative versions in the constant exponent case.  In \cite[Theorem 2.3]{MR3092729}, the authors give a quantitative expression for the sharp reverse H\"{o}lder exponent in terms of the so-called Fujii-Wilson  $A_\infty$ constant, namely, $r=1+\frac{1}{\tau_d [w]'_{A_\infty}-1}$. In \cite[Theorem 1.1]{MR2990061}, the authors give an different proof of the sharp reverse H\"{o}lder exponent in spaces of homogeneous type, giving the sharp quantitative expression for $\tau_d$, namely, $\tau_d = 2^{d+1}$. In~\cite[Theorem~3.2]{cruzuribe2017extrapolation}, the author proved a weaker version with a constant in terms of the $A_p$ constant.  See Lemma~\ref{AInftyEquivRH} below.

\begin{remark}
When $\pp = p$ is constant, inequality \eqref{ineq:NormRH} reduces to inequality~\eqref{eqn:calAp-holder}.  To see this, note that $|Q|^{\frac{1}{p_Q}}\approx \|\chi_Q\|_{L^\pp(\R^n)}$, and when $\pp$ is constant this becomes $|Q|^{\frac{1}{p}}$. (See Lemma~\ref{CharFunctionNormIneq} below.)  We note, however, that the constant $C_\pp$ doesn't reduce to the one gotten in the constant exponent case. See Remark \ref{rem:p(.)=pConstant} for details.
\end{remark}

\medskip

As a consequence of the reverse H\"older inequality, we can prove that the class of $\calA_\pp$ weights is both right and left open.

\begin{corollary} \label{cor:scalar-calAppRightOpen}
Given $\pp \in \Pp(\R^n)\cap LH(\R^n)$ with $p_+<\infty$ and a  weight  $w \in \calA_\pp$, there exists $r >1$ such that for all $s\in [1,r]$, $w\in \calA_{s\pp}$.
\end{corollary}

\begin{corollary} \label{cor:scalar-calAppLeftOpen}
Given $\pp \in \Pp(\R^n)\cap LH(\R^n)$ with $p_->1$ and a  weight  $w \in \calA_\pp$, there exists $r >1$ such that for all $s\in [1,r]$, if $\cqp = s \cpp$, then $w \in \calA_{\qp}$.
\end{corollary}

We will actually derive these results as special cases of the corresponding results for  matrix weights in the variable exponent setting.
Recall that the study of matrix $A_p$ weights began with Nazarov, Treil, and Volberg in the 1990s (see \cite{treil_wavelets_1997,nazarov_hunt_1996}).  They defined  matrix  $A_p$ condition and proved bounds for the Hilbert transform on matrix weighted $L^p$ spaces.  The definition of matrix  $A_p$ was originally stated in terms of norm functions, but Roudenko~\cite{roudenko_matrix-weighted_2002} gave an equivalent definition of $A_p$ strictly in terms of matrices. Given $1 < p <\infty$, a  matrix weight $V$ is a $d\times d$ matrix function that is positive definite (and so invertible) almost everywehere.  It is a matrix $A_p$ weight if 
\[ [V]_{A_p} : = \sup_Q \dashint_Q \left( \dashint_Q |V^{\frac{1}{p}}(x) V^{-\frac{1}{p'}}(y)|_{\op}^{p'} \, dy \right)^{\frac{p}{p'}} \, dx <\infty.\]
When $d=1$ this reduces the the classical scalar $A_p$ condition.  These matrix weights have some fine properties in common with scalar weights: for example, they are nested, with matrix $A_p$ contained in matrix $A_q$ if $q>p$. (See Goldberg~\cite[Proposition~5.5]{goldberg_matrix_2003}.)   However, one property fails: they are not left open.  Bownik~\cite[Corollary~4.3]{MR1857041} gave an explicit example of a matrix weight $V\in A_2$ that is not in $A_p$ for any $p<2$.  (The introduction of this paper also gives a good summary of the properties of scalar weights that do and do not extend to matrix weights.)

As in the scalar setting, and with the same motivation, we define an alternative class to matrix $A_p$, which we again denote by $\calA_p$. Beginning with Nazarov, Treil and Volberg, the norm in $L^p(V)$ was written
\[ \|f\|_{L^p(V)} = \left( \int_{\R^n} |V^{\frac{1}{p}}(x) f(x)|^p\, dx \right)^{\frac{1}{p}}.\]
If we instead write the norm as 
\[ \|f\|_{L^p(W)} = \|Wf\|_{L^p(\R^n)} = \left( \int_{\R^n} |W(x) f(x)|^p\, dx \right)^{\frac{1}{p}},\]
then we are led to the following alternative definition, first adopted by Bownik and the first author~\cite{bownik_extrapolation_2022}.  
Given $1 < p  <\infty$,  a matrix weight $W$ is a matrix $\calA_p$ weight if 
\[ [W]_{\calA_p} : = \sup_Q \left( \dashint_Q \left( \dashint_Q |W(x)W^{-1}(y)|_{\op}^{p'} \, dy \right)^{p/p'} \, dx \right)^{\frac{1}{p}} <\infty.\]
The relationship between matrix $A_p$ and $\calA_p$ is the same as in the scalar setting: given a matrix weight $V \in A_p$, $W= V^{\frac{1}{p}} \in \calA_p$, with $[W]_{\calA_p} = [V]_{A_p}^{\frac{1}{p}}$. The $\calA_p$ classes are not nested:  in $\R$, consider a diagonal matrix whose diagonal entries are $|x|^{-\frac{1}{2}}$.  However, in contrast to Roudenko's $A_p$ classes, they are both right and left open.  This is a consequence of our results below.

This definition can also be rewritten using $L^p$ norms, i.e.,
\[ [W]_{\calA_p} = \sup_Q |Q|^{-1} \big\| \big\| |W(x)W^{-1}(y)|_{\op}\chi_Q(y)\big\|_{L^{p'}_y(\R^n)} \chi_Q(x)\big\|_{L^p_x(\R^n)}.\]
This definition generalizes to variable Lebesgue spaces: given an exponent function $\pp$, an invertible matrix weight $W$ is a matrix $\calA_\pp$ weight if 
\[ [W]_{\calA_\pp} : = \sup_Q |Q|^{-1}\big\| \big\| |W(x)W^{-1}(y)|_{\op}\chi_Q(y)\big\|_{\Lcpp_y(\R^n)} \chi_Q(x)\big\|_{\Lpp_x(\R^n)} <\infty.\]
We introduced the class $\calA_\pp$ in~\cite{ConvOpsOnVLS}, where we used it to study the boundedness of averaging and convolution operators on $L^p(W)$. 
Here, as an application of Theorem~\ref{thm:NormRH}, we show that they are both right and left open.
\begin{theorem}\label{thm:calAppRightOpen}
Given $\pp \in \Pp(\R^n)\cap LH(\R^n)$ with $p_+<\infty$ and a matrix weight $W: \R^n\to \calS_d$, if $W \in \calA_\pp$, then there exists $r >1$ such that for all $s\in [1,r]$, $W \in \calA_{s\pp}$.
\end{theorem}

\begin{theorem}\label{thm:calAppLeftOpen}
Given $\pp \in \Pp(\R^n)\cap LH(\R^n)$ with $p_->1$ and a matrix weight $W: \R^n\to \calS_d$, suppose $W \in \calA_\pp$. Then there exists $r >1$ such that for all $s\in [1,r]$, if $\cqp = s \cpp$, then $W \in \calA_{\qp}$.
\end{theorem}

\medskip

The remainder of this paper is organized as follows.
In Section \ref{sec:Prelim}, we state the relevant definitions and lemmas about variable Lebesgue spaces. All of these results are known; we keep very careful track of the constants, and in one case, Lemma~\ref{lem:CubeNormMeasureEquiv}, we give a new proof in order to clearly determine the constants (or more precisely, what the constants depend on).  In Section \ref{sec:ScalarWeightLemmas}, we prove several technical lemmas about scalar $\calA_\pp$ weights. Again, these results are not new; they appeared previously in ~\cite{MR2927495}.  However, their proofs were qualitative and since we need to have quantitative estimates on the constants, we give detailed proofs.  In Section \ref{sec:NormRH}, we prove Theorem \ref{thm:NormRH}.  The proof is extremely technical and depends on a sharp version of the reverse H\"older inequality for scalar $A_p$ weights:  see Lemma~\ref{AInftyEquivRH}.  In both Sections~\ref{sec:ScalarWeightLemmas} and~\ref{sec:NormRH} we make repeated normalization arguments to control the constants.  In particular, our proofs initially seem to show that the constants depend on $\|w\chi_{Q_0}\|_{L^\pp(\R^n)}$, where $Q_0$ is a fixed cube centered at the origin, but we are able to remove this dependence.  This plays an important role in the proofs for matrix weights: see Lemma~\ref{lem:NormAvgUnifBound1}.  We note that in the proof of the weighted norm inequalities for the maximal operator in~\cite{MR2927495}, this same dependence appears, but our normalization argument can be used to remove it, so that the constant only depends on the $\calA_\pp$ characteristic of the weight.  This leads to the problem of determining the sharp dependence on this constant.  
In Section \ref{sec:MatrixWts}, we give the relevant definitions for matrix weights and prove several key lemmas related to averaging operators. Lastly, in Section \ref{sec:LeftRightOpen}, we prove Theorems \ref{thm:calAppRightOpen} and \ref{thm:calAppLeftOpen}.

Throughout this paper, we will use the following notation. We use $n$ to denote the dimension of the Euclidean space $\R^n$, and $d$ will denote the dimension of matrix and vector-valued functions. When we use cubes $Q$, we assume their sides are parallel to the coordinate axes. Given two values $A$ and $B$, we will write $A \lesssim B$ if there exists a constant $c$ such that $A \leq cB$. We write $A \approx B$ if $A \lesssim B$ and $B \lesssim A$. We will often indicate the parameters constants depend on by writing, for example, $C(n,\pp)$.  By a (scalar) weight $w$ we mean a non-negative, locally integrable function such that $0<w(x)<\infty$ almost everywhere.

\section{Preliminaries}\label{sec:Prelim}
In this section we give the  basic definitions and lemmas about variable Lebesgue spaces. We refer the reader to \cite{cruz-uribe_variable_2013,diening_lebesgue_2011} for the proofs of many of these results, as well as a thorough treatment of variable Lebesgue spaces.

An exponent function is a Lebesgue measurable function $\pp: \R^n \to [1,\infty]$. Denote the collection of all exponent functions on $\R^n$ by $\Pp(\R^n)$. Given a set $E\subseteq \R^n$, define
\[ p_+(E) = \esssup_{x\in E} p(x), \quad \text{and} \quad  p_-(E) =\essinf_{x\in E} p(x). \]
For brevity, we will write  $p_+=p_+(\R^n)$ and $p_-=p_-(\R^n)$. If $0<|E|<\infty$, define the harmonic mean of $\pp$ on $E$, denoted $p_E$, by 
\[ \frac{1}{p_E} = \dashint_E \frac{1}{p(x)}\,dx.\]
Define the conjugate exponent function to $\pp$, denoted $\cpp$, by 
\[ \frac{1}{p(x)} + \frac{1}{p'(x)} = 1,\]
for all $x\in \R^n$, where we use the convention that $1/\infty=0$.\\

Given $\pp\in \Pp(\R^n)$, define the modular associated with $\pp$ by 
\[\rho_\pp(f) = \int_{\R^n\bk \Omega_\infty} |f(x)|^{p(x)} dx + \|f\|_{L^\infty(\Omega_\infty)},\]
where $\Omega_\infty = \{x\in \R^n: p(x) = \infty\}$. Define $\Lpp(\R^n)$ to be the collection of Lebesgue measurable functions $f:\R^n \to \R$ such that
\[\|f\|_{\Lpp(\R^n)}:= \inf\{\lambda>0: \rho_\pp(f/\lambda)\leq 1\}<\infty.\]
If $f$ depends on two variables, $x$ and $y$, we specify which variable the norm is taken with respect to with subscripts, e.g., $\Lpp_x$ and $\Lpp_y$.

Given a weight $w$, define $\Lpp(w)$ to be the collection of Lebesgue measurable functions $f: \R^n\to \R$ such that  
\[\|f\|_{\Lpp(w)} := \|wf\|_{\Lpp(\R^n)}<\infty.\]

We now state some important lemmas  about variable Lebesgue spaces.

\begin{lemma}\cite[Corollary 2.23]{cruz-uribe_variable_2013}\label{cor:ModNormEquiv}
Let $\pp \in \Pp(\R^n)$ with $p_+<\infty$. If $\|f\|_{\Lpp(\R^n)}>1$, then 
\[\rho_\pp(f)^{\frac{1}{p_+}} \leq \| f\|_{\Lpp(\R^n)}\leq \rho_\pp(f)^{\frac{1}{p_-}}.\]
If $0 <\| f\|_{\Lpp(\R^n)} \leq 1$, then
\[\rho_\pp(f)^{\frac{1}{p_-}} \leq \| f\|_{\Lpp(\R^n)} \leq \rho_\pp(f)^{\frac{1}{p_+}}.\]
\end{lemma}

\begin{lemma}\cite[Proposition 2.21]{cruz-uribe_variable_2013}\label{prop:NormalizedMod}
Given $\pp \in \Pp(\R^n)$, if $f \in \Lpp(\R^n) $with $\| f\|_{\Lpp(\R^n)} >0$, then $\rho_\pp(f/\|f\|_{\Lpp(\R^n)} )\leq1$. Furthermore, $\rho_\pp(f/\|f\|_{\Lpp(\R^n)})=1$ for all non-trivial $f\in \Lpp(\R^n)$ if and only if $p_+(\R^n\bs\Omega_\infty)<\infty$. 
\end{lemma}

\begin{lemma}\cite[Theorem 2.26]{cruz-uribe_variable_2013}\label{Holder}
Given $\pp \in \Pp(\R^n)$, for all $f \in \Lpp(\R^n)$ and $g \in \Lcpp(\R^n)$, $fg \in L^1(\R^n)$ with
\[\int_{\R^n} |f(x) g(x) |dx \leq K_\pp \|f\|_{\Lpp}\|g\|_{\Lcpp},\]
where $K_\pp$ is a constant depending only on $\pp$. If $ p_+ <\infty$, then $K_\pp \leq 3$; if $1<p_-\leq p_+<\infty$, then $K_\pp\leq 2$.
\end{lemma}

\begin{lemma}\cite[Corollary 2.28]{cruz-uribe_variable_2013}\label{lem:GeneralizedHolder}
Given $\pp, \qp\in \Pp(\R^n)$ define $\rp \in \Pp(\R^n)$ by
\[\frac{1}{q(x)} = \frac{1}{p(x)} + \frac{1}{r(x)}\]
for $x\in \R^n$. Then there exists a constant $K=K_{\pp/\qp}+1$ such that for all $f \in \Lpp(\R^n)$ and $g \in \Lrp(\R^n)$, $fg \in \Lqp(\R^n)$ with
\[\| fg\|_{\Lqp(\R^n)} \leq K \| f\|_{\Lpp(\R^n)} \| g\|_{\Lrp(\R^n)}.\]
\end{lemma}

Next, we define log-H\"{o}lder continuity, which is an important hypothesis for our results.
\begin{definition}\label{LH:def}
A function $\up: \R^n \to \R$ is locally log-H\"{o}lder continuous, denoted by $\up \in LH_0(\R^n)$, if there exists a constant $C_0$ such that for all $x,y\in \R^n$ with $|x-y|<1/2$,
\begin{align}\label{LH0}
|u(x)-u(y)| \leq \frac{C_0}{-\log(|x-y|)}.
\end{align}
We say that $\up$ is log-H\"{o}lder continuous at infinity, denoted $\up \in LH_\infty(\R^n)$, if there exist constants $C_\infty$ and $u_\infty$ such that for all $x\in \R^n$,
\begin{align}\label{LHinfty}
|u(x)-u_\infty| \leq \frac{C_\infty}{\log(e+|x|)}.
\end{align}
If $\up$ is log-H\"{o}lder continuous locally and at infinity, we denote this by $\up\in LH(\R^n)$. 
\end{definition}
\begin{remark}\label{rem:ReciprocalsInLH}
When $p_+=\infty$, it is more natural to assume that $1/\pp\in LH(\R^n)$, instead of assuming $\pp \in LH(\R^n)$. But it is well-known (see~\cite[Proposition 2.3]{cruz-uribe_variable_2013}) that if $p_+<\infty$, then $\pp\in LH(\R^n)$ if and only if $1/\pp \in LH(\R^n)$. 
\end{remark}

The following lemma connects local log-H\"{o}lder continuity to a condition on cubes. We refer to this lemma as the Diening condition, as he first proved this result.  
\begin{lemma}\cite[Lemma 3.24]{cruz-uribe_variable_2013},\cite[Lemma 4.1.6]{diening_lebesgue_2011}\label{lem:DieningCondition}
Let $\pp \in \Pp(\R^n)$ with $p_+<\infty$. If $\pp \in LH_0(\R^n)$, then there exists a constant $C_D$ such that for all cubes $Q\subset \R^n$, 
\[|Q|^{p_-(Q)-p_+(Q)} \leq C_D.\]
In fact, we may take $C_D= \max\{(2\sqrt{n})^{n(p_+-p_-)}, \exp(C_0 (1+\log_2\sqrt{n}))\}$.
\end{lemma}

The following lemma is a variant of the Diening condition. The proof follows the same arguments with very little modification. We leave the details to the reader.

\begin{lemma}\label{lem:GeneralDieningCondition}
Let $\pp \in \Pp(\R^n)$ with $p_+<\infty$. If $\pp \in LH_0(\R^n)$, then for all cubes $Q$ and all $x, y \in Q$, 
\begin{align}
|Q|^{-|p(x)-p(y)|} \leq C_D.
\end{align}
Moreover, if $|Q|\leq 1$, then for any $x\in Q$, 
\[|Q|^{-1}\leq C_D^{\frac{1}{p_-}}|Q|^{-\frac{p(x)}{p_Q}} \qquad \text{and} \qquad 
|Q|^{-\frac{p(x)}{p_Q}}\leq C_D^{\frac{1}{p_-}} |Q|^{-1}.\]
\end{lemma}

Our next result lets us relate the norm of the characteristic function of a cube to its measure.  The result is stated in terms of the $\calA_\pp$ condition, given in Definition~\ref{def:scalarApp} below; we state the special case needed here.  Given $\pp \in \Pp(\R^n)$, we say that $1\in \calA_\pp$ if
\[ [1]_{\calA_\pp} = \sup_Q |Q|^{-1} \|\chi_Q\|_{L^\pp(\R^n)} \|\chi_Q\|_{L^\cpp(\R^n)} < \infty. \]

\begin{remark} \label{rem:[1]Finite}
If $p_+<\infty$ and $\pp \in LH(\R^n)$, then $1\in \calA_\pp$; but this condition is strictly weaker than log-H\"older continuity.  See~\cite[Proposition~4.57, Example~4.59]{cruz-uribe_variable_2013}; note that there this condition is referred to as the $K_0$ condition.  
\end{remark}

\begin{lemma}\cite[Proposition 3.8]{TroyThesis}\label{CharFunctionNormIneq}
Given $\pp\in \Pp(\R^n)$ such that $p_+<\infty$ and $1\in \calA_\pp$, for any cube $Q$, 
\[\frac{1}{2 K_\pp}|Q|^{\frac{1}{p_Q}} \leq \| \chi_Q\|_{\Lpp(\R^n)} \leq 4K_\pp^2[1]_{\calA_\pp}|Q|^{\frac{1}{p_Q}}.\]
In particular, this holds if $\pp \in LH(\R^n)$.
\end{lemma}
\begin{lemma}\cite[Corollary 4.5.9]{diening_lebesgue_2011}\label{lem:CubeNormMeasureEquiv}
Let $\pp \in \Pp(\R^n)$ with $1/\pp \in LH(\R^n)$. Then there exist constants $D_1$ and $D_2$ such that for every cube $Q$ with $|Q|\geq 1$,
\[ |Q|^{\frac{1}{p_\infty}} \leq D_1 \| \chi_Q \|_{\Lpp(\R^n)},\]
and 
\[ \|\chi_Q\|_{\Lpp(\R^n)} \leq D_2 |Q|^{\frac{1}{p_\infty}},\]
The constants $D_1$ and $D_2$ depend only on $n$, $\pp$, and the log-H\"{o}lder constant of $1/\pp$.
\end{lemma}
We provide the proof of this lemma to track the constants; it is somewhat simpler than the one given in~\cite{diening_lebesgue_2011}.
\begin{proof}
Fix a cube $Q$ with $|Q|\geq 1$. By Lemma \ref{CharFunctionNormIneq}, it will suffice to show that $|Q|^{\frac{1}{p_Q}} \approx |Q|^{\frac{1}{p_\infty}}$, or equivalently, 
\begin{equation} \label{eqn:equiv}
    \left| \frac{1}{p_Q} - \frac{1}{p_\infty}\right| \log |Q| \leq C,
\end{equation} 
where $C$ depends on $n$ and the log-H\"{o}lder constant of $1/\pp$.  By the definition of $p_Q$ and the $LH_\infty$ condition on $1/\pp$,
\begin{equation*}
\left| \frac{1}{p_Q} - \frac{1}{p_\infty}\right| = \left| \dashint_Q \frac{1}{p(x)}-\frac{1}{p_\infty}\, dx\right| 
		 \leq \dashint_Q \frac{D_\infty}{\log(e+|x|)} \, dx,
\end{equation*}
where $D_\infty$ is the $LH_\infty$ constant of $1/\pp$.
Let $P$ be a cube centered at the origin with $\ell(P) = \ell(Q)$, fix $R = \sqrt{n}\ell(P)$, and let $B_0 = B(0,R/2)$. Then $P \subset B_0$. 
Since the integrand increases if we move from an arbitrary cube to a cube centered at the origin with the same side length, we have

\begin{align*}
\dashint_Q \frac{D_\infty}{\log(e+|x|)} \, dx 
& \leq \dashint_P \frac{D_\infty}{\log(e+|x|)}\, dx \\
	& \leq \left(\frac{\sqrt{n}}{2}\right)^n v_n  \avgint_{B_0} \frac{D_\infty}{\log(e+|x|)}\, dx,
 \intertext{where $v_n$ is the volume of the unit ball in $\R^n$. If we convert to spherical coordinates with $r=Rs$ and $s\in (0,1)$, we get}
& = \left(\frac{\sqrt{n}}{2}\right)^n v_n \int_0^1 \frac{D_\infty s^{n-1}}{\log(e+Rs)} \, ds \\
& \leq \left(\frac{\sqrt{n}}{2}\right)^n v_n D_\infty\bigg( \int_0^{1/\sqrt{R}} \frac{ds}{\log(e+Rs)} + \int_{1/\sqrt{R}}^1 \frac{ds}{\log(e+Rs)} \bigg) \\
& \leq \left(\frac{\sqrt{n}}{2}\right)^n v_n D_\infty\bigg(\frac{1}{\sqrt{R}}+ \frac{1}{\log(e+\sqrt{R})}\bigg) \\
& \leq 2\left(\frac{\sqrt{n}}{2}\right)^n v_n D_\infty \frac{1}{\log(\sqrt{R})}\\
& = 4n\left(\frac{\sqrt{n}}{2}\right)^n v_n D_\infty \frac{1}{\log(|Q|)};
\end{align*}
in the second to last inequality we used the fact that $\log(x)\leq x$ for all $x\geq 1$.  This proves inequality~\eqref{eqn:equiv} with $C=4n\left(\frac{\sqrt{n}}{2}\right)^n v_n D_\infty$. 
\end{proof}

The following lemma allows us to replace variable exponents with constant ones, and vice versa, at the cost of a remainder term.  This result was proved in \cite[Lemma 3.26]{cruz-uribe_variable_2013} for the Lebesgue measure, but the same proof works for a general non-negative measure $\mu$. 
\begin{lemma}\label{lem:LHInftyRemainderIneq}
Let $\up: \R^n\to [0,\infty)$ be such that $\up \in LH_\infty(\R^n)$ and $0 < u_\infty <\infty$, and for $t>0$, let $R_t(x)=(e+|x|)^{-nt}$. Then for any non-negative measure $\mu$, for any set $E$ with $\mu(E)<\infty$, and any function $F$ with $0 \leq F(y)\leq 1$ for $y\in E$, 
\begin{align*}
\int_E F(y)^{u(y)} \,d\mu \leq e^{ntC_\infty} \int_E F(y)^{u_\infty} \, d\mu + \int_E R_t(y)^{u_-} \, d\mu, 
\end{align*}
and
\begin{align*}
\int_E F(y)^{u_\infty} \, d\mu \leq e^{ntC_\infty}\int_E F(y)^{u(y)} \, d\mu + \int_E R_t(y)^{u_-} \, d\mu.
\end{align*}
\end{lemma}
%

\section{Lemmas for Scalar $\calA_\pp$ Weights}\label{sec:ScalarWeightLemmas}

In this section we define the scalar $\calA_\pp$ weights and prove a number of  important, quantitative lemmas about them.  All of these results were proved in \cite{MR2927495}, but the proofs were qualitative and did not keep careful track of the constants.  Because precise estimates are necessary for our proofs in Section~\ref{sec:NormRH}, we give detailed proofs here.  
\begin{definition}\label{def:scalarApp}
Given $\pp\in \Pp(\R^n)$, define scalar $\calA_\pp$ to be the set of scalar weights $w$ such that 
\[ [w]_{\calA_\pp} : = \sup_Q |Q|^{-1}\|w\chi_Q\|_{\Lpp(\R^n)} \|w^{-1}\chi_Q\|_{\Lcpp(\R^n)} <\infty,\]
where the supremum is taken over all cubes $Q\subset \R^n$.
\end{definition}

\begin{lemma}\cite[Lemma 3.2]{MR2927495}\label{lem:Lemma3.2}
Let $\pp \in \Pp(\R^n) \cap LH(\R^n)$ and $w\in \calA_\pp$. Then, for all cubes $Q$ and all measurable sets $E\subset Q$, 
\[ \frac{|E|}{|Q|} \leq K_\pp [w]_{\calA_\pp} \frac{\|w\chi_E\|_{\Lpp(\R^n)}}{\| w\chi_Q\|_{\Lpp(\R^n)}}.\]
\end{lemma}

The original proofs of  Lemmas~\ref{lem:WtdNormDieningIneq} and~\ref{lem:AppToAInfty} yielded a constant that depended on the norm of the scalar weight $w$ on a fixed cube centered at the origin, but did not track the dependence on this quantity.  Since we needed to remove this dependence in our proof of the reverse H\"older inequality, it was necessary to  carefully track this dependence.  The first lemma is a weighted version of the Diening condition,  Lemma~\ref{lem:DieningCondition}.
\begin{lemma}\cite[Lemma 3.3]{MR2927495}\label{lem:WtdNormDieningIneq}
Given $\pp \in \Pp(\R^n) \cap LH(\R^n)$ with $p_+ <\infty$, and a weight $w\in \calA_\pp$, there exists a constant $L_1$ such that for all cubes $Q\subset \R^n$,
\[ \| w\chi_Q\|_{\Lpp(\R^n)}^{p_-(Q)-p_+(Q)}\leq L_1[w]_{\calA_\pp}^{p_+ - p_-}.\]
We may take the constant $L_1$ to be
\begin{align*}
L_1 = C(n,\pp, C_\infty) C_D \max \{ 1, \| w \chi_{Q_0}\|_{\Lpp(\R^n)}^{p_- - p_+}\},
\end{align*}
where $Q_0 = Q(0, 2e)$.
\end{lemma}

\begin{proof}
Fix a cube $Q$. Let $Q_0 = Q(0,2e)$. We first assume that $\| w\chi_{Q_0}\|_{\Lpp(\R^n)} = 1$; we will treat the general case below by homogeneity.  Further, if  $\| w\chi_Q\|_{\Lpp(\R^n)} \geq 1$, then the desired inequality is immediate, so we assume $\| w\chi_Q\|_{\Lpp(\R^n)} <1$.  We consider multiple cases according to the relative sizes of $Q$ and $Q_0$ and their relative distance to each other. 

For the first case, suppose $|Q|\leq |Q_0|$ and $\dist(Q,Q_0)\leq \ell(Q_0)$. Then $Q \subset 5Q_0$. Thus, by Lemma \ref{Holder} and the $\calA_\pp$ condition, 
\begin{align*}
|Q| & = \int_Q w(x)w^{-1}(x)\, dx \\
	& \leq K_\pp \| w\chi_Q\|_{\Lpp(\R^n)} \| w^{-1}\chi_Q\|_{\Lcpp(\R^n)}\\
	& \leq K_\pp (10e)^n \| w\chi_Q \|_{\Lpp(\R^n)} |5Q_0|^{-1} \| w^{-1} \chi_{5Q_0}\|_{\Lcpp(\R^n)}\\
	& \leq K_\pp (10e)^n [w]_{\calA_\pp} \|w\chi_Q\|_{\Lpp(\R^n)} \| w\chi_{5Q_0}\|_{\Lpp(\R^n)}^{-1}.
\end{align*}
Since $\| w\chi_{Q_0} \|_{\Lpp(\R^n)} = 1$, we have $\| w\chi_{5Q_0}\|_{\Lpp(\R^n)}^{-1}\leq \| w\chi_{Q_0}\|_{\Lpp(\R^n)}^{-1} = 1$. Thus,
\[|Q|\leq K_\pp (10e)^n [w]_{\calA_\pp} \| w\chi_Q\|_{\Lpp(\R^n)}.\]
If we rearrange terms, raise both sides to the power $p_+(Q)-p_-(Q)$, and apply Lemma~\ref{lem:DieningCondition}, we get
\begin{align*}
\| w\chi_Q\|_{\Lpp(\R^n)}^{p_-(Q)-p_+(Q)} &\leq (K_\pp(10e)^n [w]_{\calA_\pp})^{p_+(Q)-p_-(Q)}  |Q|^{p_-(Q)-p_+(Q)}\\
	& \leq C(n,\pp) C_D [w]_{\calA_\pp}^{p_+-p_-}.
\end{align*}

For the second case, suppose $|Q| \leq |Q_0|$ and $\dist(Q,Q_0) >\ell(Q_0)$. In this case, define $\tilde{Q}=Q(0,2\dist(Q,Q_0))$. Then $Q, Q_0 \subset \tilde{Q}$ and $\ell(\tilde{Q}) \leq 2 \dist(Q,0)$. Repeating the argument in the first case, with $\tilde{Q}$ instead of $5Q_0$, we get
\begin{align*}
|Q| &\leq K_\pp \| w\chi_Q \|_{\Lpp(\R^n)} \|w^{-1}\chi_Q\|_{\Lcpp(\R^n)}\\
	& \leq K_\pp |\tilde{Q}| \| w\chi_Q\|_{\Lpp(\R^n)} |\tilde{Q}|^{-1}\|w^{-1}\chi_{\tilde{Q}}\|_{\Lcpp(\R^n)}\\
	& \leq K_\pp |\tilde{Q}| [w]_{\calA_\pp} \| w\chi_Q\|_{\Lpp(\R^n)} \| w\chi_{\tilde{Q}}\|_{\Lpp(\R^n)}^{-1}.
\end{align*}
Since $Q_0 \subset \tilde{Q}$ and $\| w\chi_{Q_0}\|_{\Lpp(\R^n)}=1$, we have $\| w\chi_{\tilde{Q}}\|_{\Lpp(\R^n)}^{-1} \leq \| w\chi_{Q_0}\|_{\Lpp(\R^n)}^{-1} = 1$. Thus,
\[|Q| \leq  K_\pp [w]_{\calA_\pp} |\tilde{Q}| \| w\chi_{Q}\|_{\Lpp(\R^n)}.\]
As before, if we rearrange terms and raise both sides to the power $p_+(Q)-p_-(Q)$, we get
\begin{align*}
\| w\chi_Q\|_{\Lpp(\R^n)}^{p_-(Q)-p_+(Q)} &\leq (K_\pp [w]_{\calA_\pp})^{p_+-p_-}   |Q|^{p_-(Q)-p_+(Q)} |\tilde{Q}|^{p_+(Q)-p_-(Q)}\\
	& \leq (K_\pp [w]_{\calA_\pp})^{p_+-p_-}  C_D |\tilde{Q}|^{p_+(Q)-p_-(Q)}.
\end{align*}

We now estimate $|\tilde{Q}|^{p_+(Q)-p_-(Q)}$. Define $d_Q = \dist(Q,0)$. Since $\pp \in LH(\R^n)$,  $\pp$ is continuous on $\R^n$. Since $Q \subset \tilde{Q}$, by the continuity of $\pp$, there exists $x_1, x_2$ in the closure of $Q$ such that $p(x_1)=p_+(Q)$ and $p(x_2)=p_-(Q)$. Moreover, $|x_1|\geq d_Q$ and $|x_2|\geq d_Q$. Thus, by the $LH_\infty$ condition,
\begin{align*}
p_+(Q)-p_-(Q) & \leq |p(x_1)-p_\infty| + |p(x_2) - p_\infty|\\
	& \leq \frac{C_\infty}{\log(e+|x_1|)} + \frac{C_\infty}{\log(e+|x_2|)}\\
	& \leq \frac{2C_\infty}{\log(e+d_Q)}.
\end{align*}
Also, by our choice of $\tilde{Q}$, 
\[ |\tilde{Q}| = \ell(\tilde{Q})^n \leq (2\dist(Q,0))^n = 2^n d_Q^n\leq 2^n (e+d_Q)^n.\]
Thus,
\begin{align*}
|\tilde{Q}|^{p_+(Q)-p_-(Q)} & \leq (2^n (e+d_Q)^n)^{p_+(Q)-p_-(Q)}\\
	& \leq 2^{n(p_+-p_-)} (e+d_Q)^{2nC_\infty/\log(e+d_Q)}\\
	& = 2^{n(p_+-p_-)} e^{2nC_\infty}.
\end{align*}
Hence,
\[ \| w\chi_Q\|_{\Lpp(\R^n)}^{p_-(Q)-p_+(Q)} \leq (K_\pp [w]_{\calA_\pp})^{p_+-p_-}  C_D 2^{n(p_+-p_-)} e^{2nC_\infty} = C(n,\pp, C_\infty) C_D [w]_{\calA_\pp}^{p_+-p_-}.\]

The third case is similar to the first. Suppose $|Q|>|Q_0|$ and $\dist(Q,Q_0) \leq \ell(Q)$. Then $Q_0 \subset 5Q$. Thus, by Lemma \ref{Holder} and the $\calA_\pp$ condition,
\begin{align*}
|Q| & \leq K_\pp \| w\chi_Q \|_{\Lpp(\R^n)} \| w^{-1}\chi_Q\|_{\Lcpp(\R^n)}\\
	& \leq K_\pp |5Q|\|w\chi_Q\|_{\Lpp(\R^n)} |5Q|^{-1}\|w^{-1} \chi_{5Q}\|_{\Lcpp(\R^n)}\\
	& \leq K_\pp |5Q| [w]_{\calA_\pp}\|w\chi_Q\|_{\Lpp(\R^n)} \|w\chi_{5Q}\|_{\Lpp(\R^n)}^{-1}\\
	& = K_\pp 5^n |Q|  [w]_{\calA_\pp}\|w\chi_Q\|_{\Lpp(\R^n)} \|w\chi_{5Q}\|_{\Lpp(\R^n)}^{-1}.
\end{align*}
Since $Q_0 \subset 5Q$, we have $\| w\chi_{5Q} \|_{\Lpp(\R^n)}^{-1} \leq \| w\chi_{Q_0}\|_{\Lpp(\R^n)}^{-1} = 1$. Thus,
\[ |Q| \leq K_\pp 5^n|Q| [w]_{\calA_\pp} \| w\chi_Q\|_{\Lpp(\R^n)}.\]

Rearranging terms and raising both sides to the power $p_+(Q)-p_-(Q)$, we get
\begin{align*}
\|w\chi_Q\|_{\Lpp(\R^n)}^{p_-(Q)-p_+(Q)} &\leq (5^n K_\pp [w]_{\calA_\pp})^{p_+-p_-} \\
	& = C(n,\pp) [w]_{\calA_\pp}^{p_+-p_-}.
\end{align*}

For the last case, suppose $|Q|>|Q_0|$ and $\dist(Q,Q_0) >\ell(Q)$. Define $\tilde{Q}$ as in the second case.  If we follow a  similar argument to that in the second case, we get
\begin{align*}
|Q|& \leq K_\pp \|w\chi_Q\|_{\Lpp(\R^n)} \| w^{-1}\chi_Q\|_{\Lcpp(\R^n)}\\
	& \leq K_\pp |\tilde{Q}| \| w\chi_Q\|_{\Lpp(\R^n)} |\tilde{Q}|^{-1} \| w^{-1}\chi_{\tilde{Q}}\|_{\Lcpp(\R^n)}.
\end{align*} 
Since $Q_0 \subset \tilde{Q}$, $\| w\chi_{\tilde{Q}}\|_{\Lpp(\R^n)}^{-1} \leq \|w\chi_{Q_0}\|_{\Lpp(\R^n)}^{-1} = 1$. Thus,
\[ |Q| \leq K_\pp |\tilde{Q}| [w]_{\calA_\pp} \| w\chi_Q\|_{\Lpp(\R^n)}.\]
Rearranging terms and raising both sides to the power $p_+(Q)-p_-(Q)$, we get
\begin{align*}
\|w\chi_Q\|_{\Lpp(\R^n)}^{p_-(Q)-p_+(Q)} & \leq (K_\pp [w]_{\calA_\pp})^{p_+-p_-}  |Q|^{p_-(Q)-p_+(Q)} |\tilde{Q}|^{p_+(Q)-p_-(Q)}\\
	& \leq (K_\pp [w]_{\calA_\pp})^{p_+-p_-} C_D |\tilde{Q}|^{p_+(Q)-p_-(Q)}.
\end{align*}

As shown in the second case, $|\tilde{Q}|^{p_+(Q)-p_-(Q)}\leq 2^{n(p_+-p_-)}e^{2nC_\infty}$. Thus, 
\[\|w\chi_Q\|_{\Lpp(\R^n)}^{p_-(Q)-p_+(Q)}\leq (K_\pp [w]_{\calA_\pp})^{p_+-p_-}  C_D 2^{n(p_+-p_-)} e^{2nC_\infty}= C(n,\pp, C_\infty) C_D [w]_{\calA_\pp}^{p_+-p_-}. \]
Therefore, we have shown that in every case,
\[\|w\chi_Q\|_{\Lpp(\R^n)}^{p_-(Q)-p_+(Q)} \leq C(n,\pp, C_\infty)C_D [w]_{\calA_\pp}^{p_+-p_-}.\]
for all cubes $Q$, provided $\|w\chi_{Q_0}\|_{\Lpp(\R^n)} = 1$. 

Now suppose $\|w\chi_{Q_0}\|_{\Lpp(\R^n)} \neq 1$. Define $w_0 = w/\|w\chi_{Q_0}\|_{\Lpp(\R^n)}$. Then 
\begin{align*}
\| w\chi_Q\|_{\Lpp(\R^n)}^{p_-(Q)-p_+(Q)} & = \left\| w_0\chi_Q\right\|_{\Lpp(\R^n)}^{p_-(Q)-p_+(Q)} \| w\chi_{Q_0}\|_{\Lpp(\R^n)}^{p_-(Q)-p_+(Q)}\\
	& \leq C(n,\pp, C_\infty) C_D[w_0]^{p_+-p_-} \| w\chi_{Q_0}\|_{\Lpp(\R^n)}^{p_-(Q)-p_+(Q)}.
\end{align*}
By the homogeneity of the definition of $[w]_{\calA_\pp}$, $ [w_0]_{\calA_\pp}=[w]_{\calA_\pp}$. If $\|w\chi_{Q_0}\|_{\Lpp(\R^n)} >1$, then 
\[  \| w\chi_{Q_0}\|_{\Lpp(\R^n)}^{p_-(Q)-p_+(Q)}<1.\]
If $\|w\chi_{Q_0}\|_{\Lpp(\R^n)}<1$, then
\[\| w\chi_{Q_0}\|_{\Lpp(\R^n)}^{p_-(Q)-p_+(Q)}\leq \| w\chi_{Q_0}\|_{\Lpp(\R^n)}^{p_--p_+}.\]
Consequently,
\[\| w\chi_Q\|_{\Lpp(\R^n)}^{p_-(Q)-p_+(Q)} \leq C(n,\pp,C_\infty) C_D [w]_{\calA_\pp}^{p_+-p_-} \max\{1, \| w\chi_{Q_0}\|_{\Lpp(\R^n)}^{p_--p_+}\}.\]
\end{proof}
\begin{remark}
The choice of $Q_0 = Q(0,2e)$ is arbitrary, and we could have used any fixed cube centered at the origin in the proof of Lemma~\ref{lem:WtdNormDieningIneq}. We will  use the same estimates in the proof of Lemma~\ref{lem:AppToAInfty} below, where the choice of $Q_0$ simplified the calculations and resulting constants. We choose the same cube $Q_0$ in Lemma \ref{lem:WtdNormDieningIneq} for consistency.
\end{remark}

The next lemma connects $\calA_\pp$ weights to an $A_\infty$ condition.  We note in passing that this condition appears closely related to the boundedness of the maximal operator on weighted variable Lebesgue spaces:  see~\cite{MR3682615}.
\begin{lemma}\cite[Lemma 3.4]{MR2927495}\label{lem:AppToAInfty}
Given $\pp \in \Pp(\R^n)\cap LH(\R^n)$ with $p_+<\infty$, let $w$ be a scalar weight, and define $\bbW (\cdot)  = w(\cdot)^{\pp}$. If $w\in \calA_\pp$, then there exists a constant $L_2$ such that for all cubes $Q\subset \R^n$ and all measurable sets $E\subset Q$, 
\[\frac{|E|}{|Q|} \leq L_2 [w]_{\calA_\pp}^{1+2\frac{C_\infty p_+}{p_\infty p_-}}\left( \frac{\bbW(E)}{\bbW(Q)}\right)^{\frac{1}{p_+}},\]
where $\bbW(E) := \int_E \bbW(x) \, dx = \rho_\pp(w\chi_E)$. In fact, we may take
\[ L_2 = C(n,\pp,C_\infty) C_D^{\frac{1}{p_-}} \max\{ \| w\chi_{Q_0}\|_{\Lpp(\R^n)}^{\frac{p_-}{p_+}-1}, \| w\chi_{Q_0}\|_{\Lpp(\R^n)}^{1-\frac{p_-}{p_+}} \},\]
where $Q_0 = Q(0,2e)$.
\end{lemma}
\begin{proof}
Fix a cube $Q$ and let $E \subset Q$ be measurable. Let $Q_0 = Q(0,2e)$. First assume that $\|w\chi_{Q_0}\|_{\Lpp(\R^n)}=1$; we will treat the general case below.  We consider several cases depending on the size of $\|w\chi_Q\|_{\Lpp(\R^n)}$ and $\|w\chi_E\|_{\Lpp(\R^n)}$.

For the first case, suppose $\|w\chi_Q\|_{\Lpp(\R^n)}\leq 1$. By Lemmas \ref{lem:Lemma3.2}, \ref{cor:ModNormEquiv}, and \ref{lem:WtdNormDieningIneq}, 
\begin{align*}
\frac{|E|}{|Q|} &\leq K_\pp [w]_{\calA_\pp} \frac{\|w\chi_E\|_{\Lpp(\R^n)}}{\|w\chi_Q\|_{\Lpp(\R^n)}} \\
	& = K_\pp [w]_{\calA_\pp} \frac{\|w\chi_E\|_{\Lpp(\R^n)}}{\|w\chi_Q\|_{\Lpp(\R^n)}^{\frac{p_-(Q)}{p_+(Q)}} \| w\chi_Q\|_{\Lpp(\R^n)}^{1-(\frac{p_-(Q)}{p_+(Q)})}}\\
	& \leq K_\pp [w]_{\calA_\pp} \frac{\bbW(E)^{\frac{1}{p_+(E)}}}{\bbW(Q)^{\frac{1}{p_+(Q)}}} \|w\chi_Q\|_{\Lpp(\R^n)}^{\frac{p_-(Q)}{p_+(Q)}-1}\\
	& \leq K_\pp [w]_{\calA_\pp} \left( \frac{\bbW(E)}{\bbW(Q)}\right)^{\frac{1}{p_+(Q)}}  \|w\chi_Q\|_{\Lpp(\R^n)}^{\frac{p_-(Q)-p_+(Q)}{p_+(Q)}}\\
	& \leq K_\pp [w]_{\calA_\pp}\left( \frac{\bbW(E)}{\bbW(Q)}\right)^{\frac{1}{p_+}}  L_1^{\frac{1}{p_+(Q)}}[w]_{\calA_\pp}^{\frac{p_+-p_-}{p_+(Q)}}\\
	& \leq K_\pp [w]_{\calA_\pp} \left( \frac{\bbW(E)}{\bbW(Q)}\right)^{\frac{1}{p_+}}  L_1^{\frac{1}{p_-}} [w]_{\calA_\pp}^{\frac{p_+-p_-}{p_-}}\\
	& = K_\pp [w]_{\calA_\pp}^{\frac{p_+}{p_-}} L_1^{\frac{1}{p_-}} \left( \frac{\bbW(E)}{\bbW(Q)}\right)^{\frac{1}{p_+}}.
\end{align*}

For the second case, suppose $\|w\chi_E\|_{\Lpp(\R^n)} \leq 1 < \|w\chi_Q\|_{\Lpp(\R^n)}$. Similar to the previous case, we use Lemmas \ref{lem:Lemma3.2} and \ref{cor:ModNormEquiv}, and the fact that $\bbW(Q)\geq 1$, to get
\begin{align*}
\frac{|E|}{|Q|} & \leq K_\pp [w]_{\calA_\pp} \frac{\|w\chi_E\|_{\Lpp(\R^n)}}{\|w\chi_Q\|_{\Lpp(\R^n)}}\\
	& \leq K_\pp [w]_{\calA_\pp} \frac{ \bbW(E)^{\frac{1}{p_+(E)}}}{\bbW(Q)^{\frac{1}{p_-(Q)}}}\\
	& \leq K_\pp [w]_{\calA_\pp} \frac{ \bbW(E)^{\frac{1}{p_+(Q)}}}{\bbW(Q)^{\frac{1}{p_+(Q)}}}\\
	& \leq K_\pp [w]_{\calA_\pp} \left( \frac{\bbW(E)}{\bbW(Q)}\right)^{\frac{1}{p_+}}.
\end{align*}

For the last case, suppose $\|w\chi_Q\|_{\Lpp(\R^n)} \geq \|w\chi_E\|_{\Lpp(\R^n)} >1$. Let $\lambda = \|w\chi_Q\|_{\Lpp(\R^n)}$. Then by Lemma \ref{lem:LHInftyRemainderIneq} with $d\mu = w(\cdot)^\pp\,dx$, for all $t>0$,
\begin{align*}
\int_Q \lambda^{-p_\infty} w(x)^{p(x)} \, dx & \leq e^{ntC_\infty} \int_Q \lambda^{-p(x)}w(x)^{p(x)}\, dx + \int_Q \frac{w(x)^{p(x)}}{(e+|x|)^{tnp_-}}\, dx \\
	& = e^{ntC_\infty} \rho_\pp \left( \frac{w\chi_Q}{\lambda}\right) + \int_Q \frac{w(x)^{p(x)}}{(e+|x|)^{tnp_-}}\, dx.
\end{align*}

By Lemma~\ref{prop:NormalizedMod}, $\rho_\pp(w\chi_Q/\lambda)=1$. Now, we need to estimate the second term. For each $k\in \N$, define $Q_k = Q(0,2e^{k+1})$. Then for each $x \in Q_k \bs Q_{k-1}$, 
\[ (e+|x|)^{tnp_-} \geq |x|^{tnp_-} \geq \left( \frac{2e^{k}}{2}\right)^{tnp_-} = e^{ktnp_-}.\]

By Lemma \ref{cor:ModNormEquiv} and our assumption that $\bbW(Q_0) = \| w\chi_{Q_0}\|_{\Lpp(\R^n)} = 1$, we get
\begin{align*}
\int_{\R^n} \frac{w(x)^{p(x)}}{(e+|x|)^{tnp_-}}\, dx & \leq e^{-tnp_-}\bbW(Q_0) + \sum_{k=1}^\infty \int_{Q_k\bs Q_{k-1}}\frac{w(x)^{p(x)}}{(e+|x|)^{tnp_-}}\, dx \\
	& \leq e^{-tnp_-}+ \sum_{k=1}^\infty e^{-kntp_-}\bbW(Q_k)\, dx \\
	& = e^{-tnp_-} + \sum_{k=1}^\infty e^{-kntp_-} \|w\chi_{Q_k}\|_{\Lpp(\R^n)}^{p_+}.
\end{align*}

Observe that by Lemma \ref{lem:Lemma3.2} applied to $Q_k$ and $Q_0$, 
\begin{align*}
\| w\chi_{Q_k}\|_{\Lpp(\R^n)}  \leq K_\pp [w]_{\calA_\pp} \frac{|Q_k|}{|Q_0|} \|w\chi_{Q_0}\|_{\Lpp(\R^n)} = K_\pp [w]_{\calA_\pp} e^{kn}.
\end{align*}
Combining this with the previous estimate, we get
\begin{align*}
\int_Q \frac{w(x)^{p(x)}}{(e+|x|)^{tnp_-}}\, dx & \leq \int_{\R^n} \frac{w(x)^{p(x)}}{(e+|x|)^{tnp_-}}\, dx\\
	& \leq e^{-tnp_-} \bbW(Q_0) + \sum_{k=1}^\infty e^{-ktnp_-} (K_\pp[w]_{\calA_\pp} e^{kn})^{p_+}\\
	& = e^{-tnp_-} + (K_\pp[w]_{\calA_\pp})^{p_+} \sum_{k=1}^\infty e^{k(np_+-tnp_-)}. 
\end{align*}

Note that the sum above converges for any $t>\frac{p_+}{p_-}$. Using the formula for the sum of a geometric series, we find that $\sum_{k=1}^\infty e^{k(np_+-tnp_-)}<\frac{1}{(K_\pp[w]_{\calA_\pp})^{p_+}}$ whenever
\[ t> \frac{p_+}{p_-} + \frac{1}{np_-}\log ((K_\pp [w]_{\calA_\pp})^{p_+}+1).\]
Choose $t_1=\frac{p_+}{p_-} + \frac{1}{np_-}\log(2(K_\pp[w]_{\calA_\pp})^{p_+})$ to get
\begin{align*}
\int_Q \lambda^{-p_\infty} w(x)^{p(x)} \, dx & \leq e^{nt_1C_\infty} + e^{-t_1np_-}+(K_\pp[w]_{\calA_\pp})^{p_+} \sum_{k=1}^\infty e^{k(np_+-t_1np_-)} \\
	& \leq e^{n\frac{C_\infty p_+}{p_-}} (2(K_\pp[w]_{\calA_\pp})^{p_+})^{\frac{C_\infty}{p_-}} + \frac{e^{-np_+}}{2(K_\pp[w]_{\calA_\pp})^{p_+}} + 1\\
	& \leq C(n,\pp, C_\infty) [w]_{\calA_\pp}^{\frac{C_\infty p_+}{p_-}} +2 \\
    & \leq C(n,\pp, C_\infty) [w]_{\calA_\pp}^{\frac{C_\infty p_+}{p_-}}.
\end{align*}
Rearranging terms, we get
\begin{align}\label{ineq:AInftyDenom}
\| w\chi_Q\|_{\Lpp(\R^n)} = \lambda \geq \frac{1}{C(n,\pp,C_\infty)^{\frac{1}{p_\infty}}[w]_{\calA_\pp}^{\frac{C_\infty p_+}{p_\infty p_-}}} \bbW(Q)^{\frac{1}{p_\infty}}.
\end{align}

We now use a similar procedure, exchanging $Q$ with $E$ and $p_\infty$ with $\pp$. By Lemma~\ref{prop:NormalizedMod} and Lemma \ref{lem:LHInftyRemainderIneq}, we get
\begin{align*}
1 & = \int_E \|w\chi_E\|_{\Lpp(\R^n)}^{-p(x)} w(x)^{p(x)}\, dx\\
	& \leq e^{ntC_\infty} \int_E \|w\chi_E\|_{\Lpp(\R^n)}^{-p_\infty} w(x)^{p(x)} \, dx + \int_E \frac{w(x)^{p(x)}}{(e+|x|)^{tnp_-}}\, dx.
\end{align*}
Define the cubes $Q_k$ as before. Then for all $s>0$,
\begin{align*}
\int_E \frac{w(x)^{p(x)}}{(e+|x|)^{tnp_-}}\, dx  &\leq e^{-snp_-} \bbW(Q_0) + \sum_{k=1}^\infty \int_{Q_k\bs Q_{k-1}} \frac{w(x)^{p(x)}}{(e+|x|)^{snp_-}}\, dx\\
	& \leq e^{-snp_-} + \sum_{k=1}^\infty e^{-ksnp_-} \bbW(Q_k)\\
	& \leq e^{-snp_-} + (K_\pp[w]_{\calA_\pp})^{p_+} \sum_{k=1}^\infty e^{k(np_+ - snp_-)}. 
\end{align*}

If we find the sum the geometric series, we find that $\sum_{k=1}^\infty e^{k(np_+ - snp_-)}\leq \frac{1}{4(K_\pp[w]_{\calA_\pp})^{p_+}}$ whenever
\[s \geq \frac{p_+}{p_-} + \frac{1}{np_-} \log(4(K_\pp[w]_{\calA_\pp})^{p_+})+1).\]
If we choose $s_1 = \frac{p_+}{p_-} + \frac{1}{np_-}\log(5(K_\pp[w]_{\calA_\pp})^{p_+})$, we get
\[(K_\pp[w]_{\calA_\pp})^{p_+} \sum_{k=1}^\infty e^{k(np_+ - snp_-)}\leq \frac{1}{4}.\]
Also,
\begin{align*}
e^{-s_1np_-} = e^{-np_+} e^{-\log(5(K_\pp[w]_{\calA_\pp})^{p_+}} 
	 = e^{-np_+} \frac{1}{5(K_\pp [w]_{\calA_\pp})^{p_+}}
	 \leq \frac{1}{5},
\end{align*}
and so
\[ \int_E \frac{w(x)^{p(x)}}{(e+|x|)^{s_1 np_-}}\, dx \leq e^{-s_1np_-} + (K_\pp[w]_{\calA_\pp})^{p_+} \sum_{k=1}^\infty e^{k(np_+ - snp_-)} \leq \frac{1}{5}+\frac{1}{4} < \frac{1}{2}.\]
Furthermore, 
\[e^{ns_1C_\infty} = e^{\frac{nC_\infty p_+}{p_-}}(5[w]_{\calA_\pp})^{\frac{nC_\infty p_+}{np_-}} 
= C(n,\pp,C_\infty)[w]_{\calA_\pp}^{\frac{C_\infty p_+}{p_-}}.\]

If we combine all of these estimates, we get
\begin{align*}
1 & \leq e^{ns_1 C_\infty} \int_E \|w\chi_E\|_{\Lpp(\R^n)}^{-p_\infty} w(x)^{p(x)}\, dx + \int_E \frac{w(x)^{p(x)}}{(e+|x|)^{s_1 np_-}}\, dx \\
	& \leq C(n,\pp,C_\infty) [w]_{\calA_\pp}^{\frac{C_\infty p_+}{p_-}} \|w\chi_E\|_{\Lpp(\R^n)}^{-p_\infty} \bbW(E) +\frac{1}{2}. 
\end{align*}

If we now rearrange terms, we get
\begin{align}\label{ineq:AInftyNumer}
\| w\chi_E \|_{\Lpp(\R^n)} \leq 2^{\frac{1}{p_\infty}} C(n,\pp,C\infty)[w]_{\calA_\pp}^{\frac{C_\infty p_+}{p_\infty p_-}} \bbW(E)^{\frac{1}{p_\infty}}.
\end{align}

Therefore, by Lemma \ref{lem:Lemma3.2} and inequalities \eqref{ineq:AInftyDenom} and \eqref{ineq:AInftyNumer}, we get
\begin{align*}
\frac{|E|}{|Q|} & \leq K_\pp [w]_{\calA_\pp} \frac{\|w\chi_E\|_{\Lpp(\R^n)}}{\| w\chi_Q\|_{\Lpp(\R^n)}}\\
	& \leq K_\pp [w]_{\calA_\pp}2^{\frac{1}{p_\infty}} C(n,\pp,C_\infty) [w]_{\calA_\pp}^{\frac{C_\infty p_+}{p_\infty p_-}}[w]_{\calA_\pp}^{\frac{C_\infty p_+}{p_\infty p_-}} \left( \frac{\bbW(E)}{\bbW(Q)}\right)^{\frac{1}{p_\infty}}\\
	& \leq C(n,\pp,C_\infty) [w]_{\calA_\pp}^{1+ 2\frac{C_\infty p_+}{p_\infty p_-}} \left( \frac{\bbW(E)}{\bbW(Q)}\right)^{\frac{1}{p_+}}.
\end{align*}

Thus, we have shown that in all cases that if $\|w\chi_{Q_0}\|_{\Lpp(\R^n)} = 1$, we get for all cubes $Q$ and measurable sets $E\subseteq Q$, 
\[\frac{|E|}{|Q|} \leq C(n,\pp, C_\infty)C_D^{\frac{1}{p_-}}[w]_{\calA_\pp}^{1+2\frac{C_\infty p_+}{p_\infty p_-}} \left( \frac{\bbW(E)}{\bbW(Q)}\right)^{\frac{1}{p_+}}.\]

Now suppose $\|w\chi_{Q_0}\|_{\Lpp(\R^n)}\neq 1$. Define $w_0= w/\|w\chi_{Q_0}\|_{\Lpp(\R^n)}$ and $\bbW_0(\cdot)$ by $\bbW_0 (E) = \int_E w_0(x)^{p(x)}\, dx$. Observe that if $\|w\chi_{Q_0}\|_{\Lpp(\R^n)} <1$, then by Lemma \ref{cor:ModNormEquiv},
\begin{equation*}
\frac{\bbW_0(E)}{\bbW_0(Q)}
=  \frac{\int_E w_0(x)^{p(x)}}{\int_Q w_0(x)^{p(x)}}\\
 \leq \frac{\|w\chi_{Q_0}\|_{\Lpp(\R^n)}^{-p_+}\bbW(E)}{\|w\chi_{Q_0}\|_{\Lpp(\R^n)}^{-p_-} \bbW(Q)}\\
 = \|w\chi_{Q_0}\|_{\Lpp(\R^n)}^{p_- -p_+} \frac{\bbW(E)}{\bbW(Q)}.
\end{equation*}
Similarly, if $\|w\chi_{Q_0}\|_{\Lpp(\R^n)}>1$, then
\[\frac{\bbW_0(E)}{\bbW_0(Q)} \leq \|w\chi_{Q_0}\|_{\Lpp(\R^n)}^{p_+-p_-} \frac{\bbW(E)}{\bbW(Q)}.\]
Since $[w_0]_{\calA_\pp} = [w]_{\calA_\pp}$, we get for all cubes $Q$ and measurable sets $E\subseteq Q$, 
\begin{align*}
 \frac{|E|}{|Q|}& \leq C(n,\pp,C_\infty)C_D^{\frac{1}{p_-}}[w_0]_{\calA_\pp}^{1+2\frac{C_\infty p_+}{p_\infty p_-}}\left( \frac{\bbW_0(E)}{\bbW_0(Q)}\right)^{\frac{1}{p_+}} \\
	& \leq C(n,\pp,C_\infty)C_D^{\frac{1}{p_-}}[w]_{\calA_\pp}^{1+2\frac{C_\infty p_+}{p_\infty p_-}}\max\{\|w\chi_{Q_0}\|_{\Lpp(\R^n)}^{\frac{p_-}{p_+}-1}, \|w\chi_{Q_0}\|_{\Lpp(\R^n)}^{1-\frac{p_-}{p_+}}\} \left( \frac{\bbW(E)}{\bbW(Q)}\right)^{\frac{1}{p_+}}.
\end{align*}
This completes the proof.
\end{proof}
\begin{remark}
We note in passing that there is a small loss of information in our proof when we pass back to the constant exponent case.  
If $\pp = p$ is constant, then $p_- = p_\infty = p_+=p$, $K_\pp=1$, $C_D = 1$, $C_\infty = 0$, and $L_1 = 1$. The constants in the first case become 
\[K_\pp [w]_{\calA_\pp}^{\frac{p_+}{p_-}} L_1^{\frac{1}{p_-}} = [w]_{\calA_\pp}.\]
For the arguments in the last case, inequality \eqref{ineq:AInftyDenom} simplifies to 
\[\|w\chi_Q\|_{L^p(\R^n)} \geq \frac{1}{3^{\frac{1}{p}}} \bbW(Q)^{\frac{1}{p}}.\]
Inequality \eqref{ineq:AInftyNumer} simplifies to 
\[\|w\chi_E\|_{L^p(\R^n)} \leq 2^{\frac{1}{p}} \bbW(E)^{\frac{1}{p}}.\]
Thus, our argments in the last case give
\[ \frac{|E|}{|Q|} \leq [w]_{\calA_\pp} 6^{\frac{1}{p}} \left( \frac{\bbW(E)}{\bbW(Q)}\right)^{\frac{1}{p}}.\]
If we use the proof in \cite[Lemma 2.5]{cruzuribe2017extrapolation} in the constant exponent setting, we would expect the constant to be $[w]_{\calA_p}$ instead of $6^{\frac{1}{p}}[w]_{\calA_p}$.  We suspect that this constant can be eliminated (or at least reduced) by a more careful choice of constants in the argument, but for our purposes the extra work did not seem necessary.
\end{remark}


\section{The Reverse H\"{o}lder Inequality in Variable Lebesgue Spaces}\label{sec:NormRH}

In this section we prove Theorem~\ref{thm:NormRH}.  The proof is substantially more difficult than the proof of the classical reverse H\"older inequality for Muckenhoupt $A_p$ weights.  It requires a careful modular estimate that depends on the size of the cube and the size of $w$ on the cube.   In order to get the final constants, we must keep very careful track of the constants and then apply a second homogenization argument.  Also, at the heart of the proof we need to apply the classical reverse H\"older inequality to estimate the modular of $w(\cdot)^{\qp} = w(\cdot)^{r\pp}$.  We can do so as a  consequence of Lemma \ref{lem:AppToAInfty}.   We state the exact version we  need for our proof.
\begin{lemma}\label{AInftyEquivRH}
Let $v: \R^n\to [0,\infty)$ be a weight. Suppose there exist constants $\delta, C_1 >0$ such that for every cube $Q$ and every measurable set $E \subseteq Q$, 
\begin{equation}\label{ineq:AInfty1}
\frac{|E|}{|Q|} \leq C_1 \left( \frac{v(E)}{v(Q)}\right)^{\delta}.
\end{equation}
Then, there exists an exponent $r>1$ such that for every cube $Q$,
\begin{align}\label{ineq:AInftyRH}
\dashint_Q v(x)^r\, dx \leq 2 \left( \dashint_Q v(x)\, dx\right)^r.
\end{align}
In particular, we can take 
\[ r \leq 1+ \frac{1}{2^{n+2+(1/\delta)}(n+1)\log(2)C_1^{1/\delta}}.\]
\end{lemma}
\begin{remark}
It is well-known that \eqref{ineq:AInfty1} is equivalent to $v \in \bigcup_{p\geq 1} A_p$. (See, for instance,~\cite[Theorem 3.1]{MR3473651}).  Inequality~\ref{ineq:AInftyRH} is then just a version of the sharp reverse H\"older inequality for $A_p$ weights proved by Hyt\"onen and P\'erez~\cite{MR3092729} and Hyt\"{o}nen, P\'erez, and Rela~\cite{MR2990061}.  This result, however, is given in terms of the so-called Fujii-Wilson $A_\infty$ constant of $v$.   Here we instead use the slightly weaker version proved in \cite[Theorem 3.2]{cruzuribe2017extrapolation}. 

By tracking the constants in the proof in \cite{cruzuribe2017extrapolation}, we see that we get the value of $r$ given. To compare our $r$ to the one obtained in the proof of \cite[Theorem 3.2]{cruzuribe2017extrapolation}, we note that in their work, they get $r=1+\epsilon$, where $\epsilon$ satisfies
\[ \epsilon a \log(a)2^p\big([v]_{A_p}^{\frac{1}{p}}\big)^p \leq \frac{1}{2},\]
for any choice of $a\geq 2^{n+1}$. Their factors of $2^p$ and $\big([v]_{A_p}^{\frac{1}{p}}\big)^{p}$ came from using the inequality in \cite[Lemma 2.5]{cruzuribe2017extrapolation}, which is 
\[ \frac{|E|}{|Q|} \leq [v]_{A_p}^{\frac{1}{p}} \left( \frac{v(E)}{v(Q)}\right)^{\frac{1}{p}},\]
under the assumption that $|Q|\leq 2|E|$. Thus, taking our inequality \eqref{ineq:AInfty1}, using the transformation $v=w^p$ and plugging in $a=2^{n+1}$, $\delta=\frac{1}{p}$, and $C_1= [v]_{A_p}^{\frac{1}{p}} = [w]_{\calA_p}$, we see that our bounds above are the same as those for the reverse H\"{o}lder exponent in \cite[Theorem 3.2]{cruzuribe2017extrapolation}. 
\end{remark}

\smallskip

\begin{proof}[Proof of Theorem \ref{thm:NormRH}]
Fix $\pp \in \Pp(\R^n)\cap LH(\R^n)$ with $p_+<\infty$ and $w\in \calA_\pp$. By Lemma~\ref{lem:AppToAInfty}, there exists a constant $L_2$ such that for all cubes $Q$ and all measurable sets $E\subseteq Q$, 
\[\frac{|E|}{|Q|} \leq L_2[w]_{\calA_\pp}^{1+2\frac{C_\infty p_+}{p_\infty p_-}}\left( \frac{\bbW(E)}{\bbW(Q)}\right)^{\frac{1}{p_+}}.\]
Consequently, by Lemma \ref{AInftyEquivRH}, there exists an exponent 
\[r=1+ \frac{1}{2^{n+2+p_+}(n+1)\log (2) \bigg(L_2 [w]_{\calA_\pp}^{1+2\frac{C_\infty p_+}{p_\infty p_-}}\bigg)^{p_+}}\]
such that for all cubes $Q$,
\begin{align}\label{ineq:ModularRH}
\dashint_Q w(x)^{rp(x)}\, dx \leq 2 \left( \dashint_Q w(x)^{p(x)}\, dx \right)^{r}.
\end{align}
We will show that this value of $r$ works for the desired reverse H\"older inequality~\eqref{ineq:NormRH}.

Let $\qp = r \pp$. By the homogeneity of the $\calA_\pp$ condition and the homogeneity of~\eqref{ineq:NormRH}, we may assume $|Q|^{-\frac{1}{p_Q}} \| w\chi_Q\|_{\Lpp(\R^n)} = 1$. Then it will suffice to show that there exists a constant $C$ such that 
\begin{equation} \label{eqn:first-hom}
    \||Q|^{-\frac{1}{q_Q}}w\chi_Q\|_{\Lqp(\R^n)}\leq C,
    \end{equation}
for all cubes $Q$. Fix a cube $Q$. We consider two cases. 

First suppose $|Q|\leq 1$. For all $x\in Q$, $\frac{q(x)}{q_Q} = \frac{p(x)}{p_Q}$. If $\||Q|^{-\frac{1}{q_Q}}w\chi_Q\|_{\Lqp(\R^n)} \leq 1$, then~\eqref{eqn:first-hom} holds with $C=1$.  Therefore, we may assume that  $\||Q|^{-\frac{1}{q_Q}}w\chi_Q\|_{\Lqp(\R^n)} >1$.  By Lemma~\ref{lem:GeneralDieningCondition} and \eqref{ineq:ModularRH}, we have that
\begin{align*}
\rho_{\qp} ( |Q|^{-\frac{1}{q_Q}}w\chi_Q) & = \int_Q |Q|^{-\frac{q(x)}{q_Q}}w(x)^{q(x)}\, dx\\
	& = \int_Q |Q|^{-\frac{p(x)}{p_Q}} w(x)^{q(x)}\, dx\\
	& \leq C_D^{\frac{1}{p_-}} \dashint_Q w(x)^{q(x)}\, dx\\
	& \leq 2C_D^{\frac{1}{p_-}} \left( \dashint_Q w(x)^{p(x)}\, dx \right)^{r}\\
	& \leq 2 C_D^{\frac{1}{p_-}} \left( C_D^{\frac{1}{p_-}} \int_Q |Q|^{-\frac{p(x)}{p_Q}}w(x)^{p(x)}\, dx\right)^{r}\\
	& = 2 C_D^{\frac{1+r}{p_-}} \rho_\pp (|Q|^{-\frac{1}{p_Q}}w\chi_Q).
\end{align*}
Therefore, by Lemma \ref{cor:ModNormEquiv} and the above inequality,  we have that
\begin{equation*}
\| |Q|^{-\frac{1}{q_Q}}w\chi_Q\|_{\Lqp(\R^n)} 
 \leq \rho_{\qp} (|Q|^{-\frac{1}{q_Q}} w\chi_Q)^{\frac{1}{q_-}}
	 \leq (2C_D^{\frac{1+r}{p_-}})^{\frac{1}{q_-}}
	 \leq 2^{\frac{1}{q_-}} C_D^{\frac{2r}{p_-q_-}}
	 \leq 2 C_D^{\frac{2}{p_-^2}}
	 \leq 2C_D^2.
\end{equation*}
This proves the first case with $C=2C_D^2$.

\medskip

For the second case, suppose $|Q|>1$. This estimate is much more technical.
Since $\pp \in LH(\R^n)$, we may apply Lemmas \ref{CharFunctionNormIneq} and \ref{lem:CubeNormMeasureEquiv} to get
\begin{align*}
\rho_{\qp}(|Q|^{-\frac{1}{q_Q}}w\chi_Q) & = \int_Q |Q|^{-\frac{q(x)}{q_Q}} w(x)^{q(x)}\, dx\\
	& = \int_Q |Q|^{-\frac{p(x)}{p_Q}} w(x)^{q(x)}\, dx \\
	& \leq (4K_\pp^2[1]_{\calA_\pp})^{p_+} \int_Q \|\chi_Q\|_{\Lpp(\R^n)}^{-p(x)} w(x)^{q(x)}\, dx\\
	& \leq (4K_\pp^2[1]_{\calA_\pp})^{p_+} \int_Q D_1^{p_+} |Q|^{-\frac{p(x)}{p_\infty}} w(x)^{q(x)}\, dx\\
	& = C(\pp,D_1)^{p_+} \int_Q |Q|^{-\frac{p(x)}{p_\infty}} w(x)^{q(x)}\, dx.
\end{align*}

Since $|Q|>1$, $|Q|^{-\frac{1}{p_\infty}} <1$. Thus, if we treat $w(\cdot)^{\qp}$ as a measure, then by Lemma \ref{lem:LHInftyRemainderIneq} we have that for all $t>0$, 
\begin{align}
\int_Q |Q|^{-\frac{p(x)}{p_\infty}} w(x)^{q(x)}\, dx & \leq e^{ntC_\infty} \int_Q |Q|^{-1} w(x)^{q(x)}\, dx + \int_Q \frac{w(x)^{q(x)}}{(e+|x|)^{ntp_-}}\, dx \nonumber\\
	& = e^{ntC_\infty} \dashint_Q w(x)^{q(x)}\, dx + \int_Q \frac{w(x)^{q(x)}}{(e+|x|)^{ntp_-}}\, dx \nonumber\\
	& \leq e^{ntC_\infty} \left( \dashint_Q w(x)^{p(x)}\,dx \right)^{r} + \int_Q \frac{w(x)^{q(x)}}{(e+|x|)^{ntp_-}}\, dx.\label{estimate1}
\end{align}

We estimate the two terms in \eqref{estimate1} separately. We start with the second term. Define $Q_k = Q(0,2e^{k+1})$ for $k\geq 0$. Then,
\begin{align*}
\int_Q \frac{w(x)^{q(x)}}{(e+|x|)^{ntp_-}}\, dx & \leq \int_{\R^n} \frac{w(x)^{q(x)}}{(e+|x|)^{ntp_-}}\, dx\\
	& \leq e^{-ntp_-} \int_{Q_0} w(x)^{q(x)}\, dx + \sum_{k=1}^\infty e^{-ktnp_-} \int_{Q_k\bs Q_{k-1}} w(x)^{q(x)}\, dx\\
	& \leq  e^{-ntp_-} \int_{Q_0} w(x)^{q(x)}\, dx + \sum_{k=1}^\infty e^{-ktnp_-} \int_{Q_k} w(x)^{q(x)}\, dx.\\
\end{align*}
Observe that since for all $k$, $|Q_k|>1$, we have $|Q_k|^{1-r}<1$. Thus, for all $k\geq 0$, 
\begin{multline*}
\int_{Q_k} w(x)^{q(x)}\, dx  = |Q_k|\dashint_{Q_k} w(x)^{q(x)}\,dx\\
	 \leq 2|Q_k|\left( \dashint_{Q_k} w(x)^{p(x)}\,dx \right)^{r}
	 = 2 |Q_k|^{1-r} \bbW(Q_k)^r
	 \leq 2 \bbW(Q_k)^r.
\end{multline*}
If we combine this with Lemma \ref{cor:ModNormEquiv}, we get
\begin{align*}
 & e^{-ntp_-} \int_{Q_0} w(x)^{q(x)}\,dx + \sum_{k=1}^\infty e^{-ktnp_-} \int_{Q_k} w(x)^{q(x)}\, dx \\
	& \qquad \leq e^{-ntp_-} 2 \bbW(Q_0)^r + 2 \sum_{k=1}^\infty e^{-kntp_-} \bbW(Q_k)^r\\
	& \qquad \leq 2e^{-ntp_-}\bbW(Q_0)^r + 2 \sum_{k=1}^\infty e^{-kntp_-} \max\{1, \|w\chi_{Q_k} \|_{\Lpp(\R^n)}^{rp_+}\}.
\end{align*}

Since $w \in \calA_\pp$, by Lemma \ref{lem:Lemma3.2}, for all $k\geq 1$,
\begin{align*}
\|w\chi_{Q_k}\|_{\Lpp(\R^n)}^{rp_+} & \leq (K_\pp [w]_{\calA_\pp})^{rp_+} \left( \frac{|Q_k|}{|Q_0|}\right)^{rp_+} \| w\chi_{Q_0}\|_{\Lpp(\R^n)}^{rp_+}\\
	& = (K_\pp [w]_{\calA_\pp})^{rp_+} e^{knrp_+} \| w\chi_{Q_0}\|_{\Lpp(\R^n)}^{rp_+}.
\end{align*}
Thus,
\begin{align}
 &  2e^{-ntp_-} \bbW(Q_0)^r + 2 \sum_{k=1}^\infty e^{-kntp_-} \max\{1, \|w\chi_{Q_k} \|_{\Lpp(\R^n)}^{rp_+}\} \nonumber\\
	& \qquad \leq 2e^{-ntp_-}\bbW(Q_0)^r + 2 \sum_{k=1}^\infty e^{-kntp_-} \max\{1, (K_\pp [w]_{\calA_\pp})^{rp_+} e^{knrp_+} \|w\chi_{Q_0}\|_{\Lpp(\R^n)}^{rp_+}\}\nonumber\\
	& \qquad \leq 2e^{-ntp_-} \bbW(Q_0)^r + 2 (K_\pp [w]_{\calA_\pp})^{rp_+} \max\{1, \|w\chi_{Q_0}\|_{\Lpp(\R^n)}^{rp_+} \} \sum_{k=1}^\infty e^{k(nrp_+-ntp_-)}.\label{NormRH:estimate2}
\end{align}
For $t> \frac{rp_+}{p_-}$, the sum above converges. By the formula for the sum of a geometric series, we see that if
\[t \geq \frac{rp_+}{p_-} + \frac{1}{np_-} \log( 2(K_\pp [w]_{\calA_\pp})^{rp_+} + 1),\]
we get 
\[2 (K_\pp [w]_{\calA_\pp})^{rp_+} \sum_{k=1}^\infty e^{k(nrp_+-ntp_-)} \leq 1.\]
Choose $t_1 = \frac{rp_+}{p_-} + \frac{1}{np_-} \log(3(K_\pp[w]_{\calA_\pp})^{rp_+} $. Then, by applying Lemma \ref{cor:ModNormEquiv} to \eqref{NormRH:estimate2}, we get
\begin{align*}
& 2e^{-nt_1p_-} \bbW(Q_0)^r + 2(K_\pp[w]_{\calA_\pp})^{rp_+} \max\{1, \|w\chi_{Q_0}\|_{\Lpp(\R^n)}^{rp_+}\} \sum_{k=1}^\infty e^{k(nrp_+-nt_1p_-)} \\
	&  \qquad \leq  2e^{-nt_1 p_-} \bbW(Q_0)^r +\max\{1, \|w\chi_{Q_0}\|_{\Lpp(\R^n)}^{rp_+}\} \\
	& \qquad \leq 2 e^{-nt_1 p_-} \max\{1,\|w\chi_{Q_0}\|_{\Lpp(\R^n)}^{rp_+}\}+\max\{1, \|w\chi_{Q_0}\|_{\Lpp(\R^n)}^{rp_+}\}\\
	& \qquad \leq 3 \max\{1,\|w\chi_{Q_0}\|_{\Lpp(\R^n)}^{rp_+}\}.
\end{align*}

Hence, we have shown that the second term in \eqref{estimate1} satisfies.
\begin{align}\label{NormRH:estimate1Part2}
\int_{Q} \frac{w(x)^{q(x)}}{(e+|x|)^{nt_1p_-}}\, dx  \leq 3 \max\{1,\|w\chi_{Q_0}\|_{\Lpp(\R^n)}^{rp_+}\}.
\end{align}

\medskip

We now estimate the first term of \eqref{estimate1}.  To estimate the constant, note that our choice of $t_1$ gives 
\begin{align}\label{NormRH:estimate1Part1a}
e^{nt_1 C_\infty} = e^{\frac{nC_\infty rp_+}{p_-}}(3(K_\pp[w]_{\calA_\pp})^{rp_+})^{\frac{C_\infty p_+}{p_-}}
= C(n,\pp,C_\infty)[w]_{\calA_\pp}^{\frac{C_\infty r (p_+)^2}{p_-}}.
\end{align}
To estimate the integral, observe that since $|Q|>1$, by Lemma \ref{lem:LHInftyRemainderIneq}, for all $s>0$,
\begin{align}
\dashint_Q w(x)^{p(x)} \, dx & = \int_Q |Q|^{-\frac{p_\infty}{p_\infty}} w(x)^{p(x)}\, dx \nonumber\\
	& \leq e^{nsC_\infty} \int_Q |Q|^{-\frac{p(x)}{p_\infty}} w(x)^{p(x)}\, dx + \int_Q \frac{w(x)^{p(x)}}{(e+|x|)^{nsp_-}}\, dx.\label{NormRH:estimate3}
\end{align}

If we define $Q_k = Q(0,2e^{k+1})$ for $k\geq 0$ and make a decomposition argument very similar to the one above, we find that if we choose
\[ s_1= \frac{p_+}{p_-} + \frac{1}{np_-} \log(2(K_\pp [w]_{\calA_\pp})^{p_+}),\]
then we can estimate the second term in \eqref{NormRH:estimate3} as follows:
\begin{align}
& \int_Q \frac{w(x)^{p(x)}}{(e+|x|)^{ns_1p_-}}\, dx \nonumber\\
	& \qquad \leq e^{-ns_1 p_-} \bbW(Q_0) + \sum_{k=1}^\infty e^{-kns_1 p_-} \bbW(Q_k)\nonumber\\
	& \qquad\leq e^{-ns_1 p_-} \bbW(Q_0) + \sum_{k=1}^\infty e^{-kns_1 p_-} \max\{1, \|w\chi_{Q_k}\|_{\Lpp(\R^n)}^{p_+}\}\nonumber\\
	& \qquad\leq e^{-ns_1 p_-} \bbW(Q_0) + (K_\pp [w]_{\calA_\pp})^{p_+} \max\{1, \|w\chi_{Q_0}\|_{\Lpp(\R^n)}^{p_+}\} \sum_{k=1}^\infty e^{-kns_1 p_- + knp_+}\nonumber\\
	& \qquad \leq e^{-ns_1 p_-} \bbW(Q_0) +\max\{1, \|w\chi_{Q_0}\|_{\Lpp(\R^n)}^{p_+}\}\nonumber\\
	& \qquad \leq 2\max\{1, \|w\chi_{Q_0}\|_{\Lpp(\R^n)}^{p_+}\}.\label{NormRH:estimate3Part2}
\end{align}

To estimate the first term in \eqref{NormRH:estimate3}, observe that
\begin{align}\label{NormRH:estimate3Part 1a}
e^{ns_1 C_\infty} =\left( e^{\frac{p_+}{p_-}} \big( 2(K_\pp [w]_{\calA_\pp})^{p_+} \big)^{\frac{1}{np_-}}\right)^{nC_\infty} = C (n,\pp, C_\infty) [w]_{\calA_\pp}^{\frac{C_\infty p_+}{p_-}}.
\end{align}
To estimate the integral in the first term, note that by Lemmas~\ref{lem:CubeNormMeasureEquiv} and~\ref{CharFunctionNormIneq}, for all $x\in Q$,
\[ |Q|^{-\frac{p(x)}{p_\infty}} \leq D_2^{p_+} \|\chi_Q\|_{\Lpp(\R^n)}^{-p(x)}\leq (2K_\pp)^{p_+} D_2^{p_+} |Q|^{-\frac{p(x)}{p_Q}}.\]
Thus,
\begin{multline}
\int_Q |Q|^{-\frac{p(x)}{p_\infty}} w(x)^{p(x)}\, dx  \leq (2K_\pp)^{p_+} D_2^{p_+} \int_Q |Q|^{-\frac{p(x)}{p_Q}} w(x)^{p(x)}\, dx  \\
	 =  (2K_\pp)^{p_+} D_2^{p_+} \rho_\pp (|Q|^{-\frac{1}{p_Q}} w\chi_Q) 
	 =(2K_\pp)^{p_+} D_2^{p_+}.\label{NormRN:estimate3Part1b}
\end{multline}
If we combine \eqref{NormRH:estimate3Part 1a}, \eqref{NormRN:estimate3Part1b} and \eqref{NormRH:estimate3Part2}, we get
\begin{align}
\dashint_Q w(x)^{p(x)}\, dx & \leq e^{ns_1 C_\infty} \int_Q |Q|^{-\frac{p(x)}{p_\infty}} w(x)^{p(x)}
+ \int_Q \frac{w(x)^{p(x)}}{(e+|x|)^{ns_1 p_-}}\, dx \nonumber \\
	& \leq C(n,\pp,C_\infty)D_2^{p_+}[w]_{\calA_\pp}^{\frac{C_\infty p_+}{p_-}}+2 \max\{1, \|w\chi_{Q_0}\|_{\Lpp(\R^n)}^{p_+}\} \nonumber \\
	& \leq C(n,\pp, C_\infty, D_2)[w]_{\calA_\pp}^{\frac{C_\infty p_+}{p_-}}\max\{1, \|w\chi_{Q_0}\|_{\Lpp(\R^n)}^{p_+}\}.\label{NormRH:estimate1Part1b}
\end{align}

If we now insert estimates \eqref{NormRH:estimate1Part1a}, \eqref{NormRH:estimate1Part1b}, and \eqref{NormRH:estimate1Part2} into \eqref{estimate1}, we get
\begin{align*}
 & \rho_{\qp}(|Q|^{-\frac{1}{q_Q}}w\chi_Q) \\
	& \quad\leq   C(\pp,D_1)\Big(e^{nt_1 C_\infty} \Big(C(n,\pp, C_\infty, D_2)[w]_{\calA_\pp}^{\frac{C_\infty p_+}{p_-}}\max\{1, \|w\chi_{Q_0}\|_{\Lpp(\R^n)}^{p_+}\}\Big)^r\\
	& \qquad \qquad+3 \max\{1,\|w\chi_{Q_0}\|_{\Lpp(\R^n)}^{rp_+}\}\Big)\\
	&\quad\leq C(n,\pp,C_\infty,D_1, D_2)^r [w]_{\calA_\pp}^{\frac{C_\infty r(p_+)^2}{p_-}}[w]_{\calA_\pp}^{r\frac{C_\infty p_+}{p_-}} \max\{1, \|w\chi_{Q_0}\|_{\Lpp(\R^n)}^{rp_+}\}\\
	& \qquad \qquad+ 3C(\pp,D_1)\max\{1,\|w\chi_{Q_0}\|_{\Lpp(\R^n)}^{rp_+}\}\\
	&\quad \leq C(n,\pp,C_\infty, D_1,D_2)^r [w]_{\calA_\pp}^{\frac{rC_\infty p_+}{p_-}(p_++1)} \max\{1,\|w\chi_{Q_0}\|_{\Lpp(\R^n)}^{rp_+}\}.
\end{align*}

Thus, by Lemma \ref{cor:ModNormEquiv} and the fact that $q_- = rp_-$, we get
\begin{align*}
|Q|^{-\frac{1}{q_Q}}\|w\chi_Q\|_{\Lqp(\R^n)} & \leq \rho_{\qp}(|Q|^{-\frac{1}{q_Q}}w\chi_Q)^{\frac{1}{q_-}}\\
 	& \leq C(n,\pp, C_\infty, D_1,D_2)^{\frac{r}{q_-}} [w]_{\calA_\pp}^{\frac{rC_\infty p_+}{q_-p_-}(p_++1)} \max\{1,\|w\chi_{Q_0}\|_{\Lpp(\R^n)}^{\frac{rp_+}{q_-}}\}\\
	& \leq C(n,\pp,C_\infty,D_1, D_2)^{\frac{1}{p_-}} [w]_{\calA_\pp}^{\frac{C_\infty p_+}{p_-^2}(p_++1)} \max\{1,\|w\chi_{Q_0}\|_{\Lpp(\R^n)}^{\frac{p_+}{p_-}}\}.
\end{align*}
This proves \eqref{eqn:first-hom} in the second case. 

Therefore, if we combine both cases, undo the homogenization, and insert the appropriate values into the definition of the constant $L_2$ from Lemma~\ref{lem:AppToAInfty}, we have that for all cubes $Q$,
\[ |Q|^{-\frac{1}{q_Q}}\| w\chi_Q\|_{\Lqp(\R^n)} \leq J |Q|^{-\frac{1}{p_Q}} \|w\chi_Q\|_{\Lpp(\R^n)},\]
where $\qp=r\pp$, 
\[r=1 + \frac{1}{C(n,\pp, C_\infty,C_D) \left([w]_{\calA_\pp}^{1+2\frac{C_\infty p_+}{p_\infty p_-}} \max\{\|w\chi_{Q_0}\|_{\Lpp(\R^n)}^{\frac{p_-}{p_+}-1}, \|w\chi_{Q_0}\|_{\Lpp(\R^n)}^{1-\frac{p_-}{p_+}}\}\right)^{p_+}},\]
and 
\[J=C(n,\pp, C_\infty, C_D, D_1, D_2) [w]_{\calA_\pp}^{\frac{C_\infty p_+}{p_-^2}(p_++1)}\max\{1,\|w\chi_{Q_0}\|_{\Lpp(\R^n)}^{\frac{p_+}{p_-}}\}.\]

To complete the proof and get the final constants, we must remove the dependence on $\|w\chi_{Q_0}\|_{\Lpp(\R^n)}$.  We can do this by a normalization argument. Define $v=w/\|w\chi_{Q_0}\|_{\Lpp(\R^n)}$. Then $\|v\chi_{Q_0}\|_{\Lpp(\R^n)} = 1$, $[v]_{\calA_\pp} = [w]_{\calA_\pp}$, and 
\begin{align}\label{ineq:NormRHNotNormalized}
 |Q|^{-\frac{1}{rp_Q}}\|v\chi_Q\|_{L^{r\pp}(\R^n)} \leq C_\pp |Q|^{-\frac{1}{p_Q}} \|v\chi_Q\|_{\Lpp(\R^n)},
\end{align}
where 
\[ r = 1+ \frac{1}{C(n,\pp,C_\infty,C_D) \left([w]_{\calA_\pp}^{1 + 2\frac{C_\infty p_+}{p_\infty p_-}}\right)^{p_+} }\]
and
\[C_\pp =C(n,\pp, C_\infty, C_D, D_1, D_2) [w]_{\calA_\pp}^{\frac{C_\infty p_+}{p_-^2}(p_++1)}.\] 
Since inequality \eqref{ineq:NormRHNotNormalized} is homogeneous, this implies that 
\[ |Q|^{-\frac{1}{rp_Q}} \|w\chi_Q\|_{L^{r\pp}(\R^n)} \leq C_\pp |Q|^{-\frac{1}{p_Q}} \| w\chi_Q\|_{\Lpp(\R^n)}.\]
This removes the dependence of the constants on $\|w\chi_{Q_0}\|_{\Lpp(\R^n)}$, and the proof is complete.
\end{proof}

\begin{remark}\label{rem:p(.)=pConstant}
In the constant exponent case, Lemma~\ref{AInftyEquivRH}, it can be shown that the exponent $r$ depends on $[w]_{\calA_p}$ to the power $1$, and the initial constant $2$ can be replaced by $1+\epsilon$ for any $\epsilon>0$ (with $r$ depending on $\epsilon$).  In our result, 
if we take $\pp = p$ to be a constant, then $p_+=p_-=p_\infty=p$, $C_\infty =0$, $C_D=1$, $D_1=2$, and $D_2=4$. By tracking the constants in the proof above in this case, we again have that the exponent depends on $[w]_{\calA_p}$ to the power $1$, but the constant in the final inequality becomes
\[ C_p = \left( 8^p (8^p+2)^r +3\right)^{\frac{1}{rp}}.\]
This constant is much larger than the value $2$, which suggests that our constant is not optimal. However, it is not clear how it can be substantially improved. 

Moreover, in the constant exponent setting, in the classical reverse H\"{o}lder inequality
\[ \dashint_Q v^{r} \, dx \leq C \left( \dashint_Q v\, dx \right)^r,\]
the constant $C$ tends to $1$ as $r\to 1$. This can be seen by tracking the constant obtained in \cite[Theorem 3.2]{cruzuribe2017extrapolation}. Their proof shows that for sufficiently small $\epsilon >0$,
\[ \dashint_Q v^{1+\epsilon} \leq \frac{1}{1-\epsilon a (\log a) 2^p [v]_{A_p}} \left( \dashint_Q v\right)^{1+\epsilon}. \]
As $\epsilon \to 0$, the constant tends to 1. This does not happen with our constant $C_p$.  It is reasonable to conjecture that the constant goes to $1$ as $r\to 1$ in the variable exponent case as well, but we are unable to prove this.
\end{remark}

\medskip

In Section~\ref{sec:MatrixWts}, in the proofs of Theorems~\ref{thm:calAppRightOpen} and~\ref{thm:calAppLeftOpen}, the right and left-openness of the matrix $\calA_\pp$ classes,  we will obtain $d$ reverse H\"{o}lder exponents and constants. To collapse them down to a uniform reverse H\"{o}lder exponent and a uniform constant, we need the following corollary to Theorem \ref{thm:NormRH}.
\begin{corollary}\label{cor:NormRHCollapsing}
Let $\pp \in \Pp(\R^n) \cap LH(\R^n)$ with $ p_+ <\infty$ and let $w$ be a scalar $\calA_\pp$ weight. Let $r$ be the exponent from Theorem \ref{thm:NormRH}. Then for all $s \in [1,r)$, when $\up = s \pp$, 
\[ |Q|^{-\frac{1}{u_Q}} \|w\chi_Q\|_{\Lup(\R^n)}\leq 32 [1]_{\calA_\vp} C_\pp |Q|^{-\frac{1}{p_Q}}\|w\chi_Q\|_{\Lpp(\R^n)},\]
for all cubes $Q$, where $\vp$ is defined by 
\begin{equation} \label{eqn:defect}
\frac{1}{\up}=\frac{1}{r\pp}+\frac{1}{\vp}.
\end{equation}
\end{corollary}
\begin{proof}
Fix $\pp \in \Pp(\R^n)\cap LH(\R^n)$, and let $r$ be the reverse H\"{o}lder exponent from Theorem~\ref{thm:NormRH}. Define $\qp = r \pp$. Fix $s \in (1,r)$, let $\up = s \pp$, and define the defect exponent $\vp$ by~\eqref{eqn:defect}.
Given any cube $Q$,  by the generalized H\"{o}lder's inequality in variable Lebesgue spaces, Lemma \ref{lem:GeneralizedHolder}, 
\[ \|w\chi_Q\|_{\Lup(\R^n)}\leq  (K_{\qp/\up}+1)\|w\chi_Q\|_{\Lqp(\R^n)}\|\chi_Q\|_{\Lvp(\R^n)}.\]
Since $\qp/\up=r/s$ is constant, $K_{\qp/\up}=1$. If we combine this with Lemma \ref{CharFunctionNormIneq} and Theorem \ref{thm:NormRH}, we get
\begin{align*}
|Q|^{-\frac{1}{u_Q}}\|w\chi_Q\|_{\Lup(\R^n)} & \leq 2|Q|^{-\frac{1}{u_Q}}\|w\chi_Q\|_{\Lqp(\R^n)}\|\chi_Q\|_{\Lvp(\R^n)}\\
	& \leq 2 |Q|^{-\frac{1}{u_Q}} \|w\chi_Q\|_{\Lqp(\R^n)} 4K_\vp^2 [1]_{\calA_\vp} |Q|^{1/v_Q}\\
	& = 8 K_\vp^2 [1]_{\calA_\vp} |Q|^{-\frac{1}{q_Q}}\|w\chi_Q\|_{\Lqp(\R^n)}\\
	& \leq 8K_\vp^2 [1]_{\calA_\vp} C_{\pp} |Q|^{-\frac{1}{p_Q}}\|w\chi_Q\|_{\Lpp(\R^n)}.
\end{align*}

Since $p_+<\infty$ and $\pp \in LH(\R^n)$, by Remark \ref{rem:ReciprocalsInLH}, $1/\pp \in LH(\R^n)$. Thus, $1/\qp$, $1/\up$, and $1/\vp \in LH(\R^n)$. Observe that
\[ \frac{1}{v(x)} =  \frac{1}{u(x)}-\frac{1}{rp(x)} = \frac{1}{p(x)}\left( \frac{1}{s}-\frac{1}{r}\right) \geq \frac{1}{p_+}\left( \frac{1}{s}-\frac{1}{r}\right) >0,\]
Thus, $v_+<\infty$ and so by Remark \ref{rem:ReciprocalsInLH}, $\vp \in LH(\R^n)$. Hence, by Remark \ref{rem:[1]Finite}, $[1]_{\calA_{\vp}}<\infty$. Additionally, $v_->1$. Thus, $K_\vp \leq 2$, and so
\[ 8K_\vp^2 [1]_{\calA_\vp} \leq 32 [1]_{\calA_\vp} <\infty.\]
\end{proof}
%

%
\section{Matrix Weights}\label{sec:MatrixWts}

In this section we prepare for the proof of Theorems~\ref{thm:calAppRightOpen} and~\ref{thm:calAppLeftOpen} by giving the definitions of matrix weights and proving some technical results.  The actual proofs are in Section~\ref{sec:LeftRightOpen} below.  Many of these definitions and results are drawn from our previous paper~\cite{ConvOpsOnVLS}, and we refer the reader there for additional information.

Let $\calS_d$ denote the collection of $d\times d$, self-adjoint, positive semi-definite matrices. The operator norm of a matrix $W$ is given by 
\[ |W|_{\op} =  \sup_{\substack{\ve\in \R^d\\ |\ve|=1}} |W\ve|.\]
The following lemmas are very useful for estimating operator norms.

\begin{lemma}{\cite[Lemma 3.2]{roudenko_matrix-weighted_2002}}\label{opNorm:equiv}
If $\{\ve_1, \ldots, \ve_d\}$ is an orthonormal basis in $\R^d$, then for any $d\times d$ matrix $B$, we have
\begin{align*}
	\frac{1}{d} \sum_{i=1}^d |B \ve_i| \leq |B|_{\op} \leq \sum_{i=1}^d |B \ve_i|.
\end{align*}
\end{lemma}
\begin{lemma}\label{SelfAdjointCommutes}
Let $U$ and $V$ be self-adjoint $d\times d$ matrices. Then $|UV|_{\op} = |VU|_{\op}$.
\end{lemma}

A matrix weight is a measurable matrix function $W: \R^n \to \calS_d$ such that $|W(\cdot)|_{\op} \in L^1_{\loc}(\R^n)$, or equivalently, the eigenvalues of $W$ are locally integrable functions. A matrix weight is invertible if it is positive definite almost everywhere, or equivalently, all its eigenvalues are positive almost everywhere.  Note that when $d=1$, matrix weights are simply locally integrable scalar weights.

We now define the variable Lebesgue spaces for the vector-valued function setting.
\begin{definition}
Given $\pp \in \Pp(\R^n)$, define $\Lpp(\R^n;\R^d)$ to be the collection of Lebesgue measurable functions $\vf : \R^n\to \R^d$ such that
\[ \| \vf\|_{\Lpp(\R^n;\R^d)} : = \| |\vf|\|_{\Lpp(\R^n)} <\infty. \]
Given a matrix weight $W:\R^n\to \calS_d$, define $\Lpp(W)$ to be the collection of Lebesgue measurable functions $\vf:\R^n\to \R^d$ such that
\[ \| \vf\|_{\Lpp(W)} : = \| W\vf\|_{\Lpp(\R^n;\R^d)} <\infty.\]
\end{definition}

We are interested in the class matrix $\calA_\pp$, which generalizes both the definition of constant exponent matrix $\calA_p$ and the definition of scalar $\calA_\pp$ in~\cite{MR2927495, MR2837636}.
\begin{definition}\label{def:MatrixApp}
Given $\pp \in \Pp(\R^n)$ and an invertible matrix weight $W : \R^n\to \calS_d$, $W$ is a matrix $\calA_\pp$ weight, denoted $W \in \calA_\pp$, if 
\[ [W]_{\calA_\pp} : = \sup_Q |Q|^{-1}\big\| \big\| |W(x)W^{-1}(y)|_{\op} \chi_Q(y)\big\|_{\Lcpp_y(\R^n)} \chi_Q(x)\big\|_{\Lpp_x(\R^n)}<\infty.\]
\end{definition}

\begin{remark}\label{rem:calAppSymmetric}
Our definition of $\calA_\pp$ if $\pp=p$ is constant reduces to the Roudenko definition (see~\cite{roudenko_matrix-weighted_2002}) if we make the change of variables $W\mapsto V^{1/p}$.  Our definition has two advantages.  First, this definition of $\calA_\pp$ has a very simple duality: $W \in \calA_\pp$ if and only if $W^{-1}\in \calA_\cpp$.  This should be contrasted with the Roudenko, where $V\in A_p$ if and only if $V^{-p'/p} \in A_{p'}$.  Second our definition, even in the constant exponent case, allows a unified definition when $p=1$ or $p=\infty$, and so allows us to define matrix $\calA_\pp$ for exponents such that $\pp$ or $\cpp$ is unbounded.
\end{remark}

There are two characterizations of matrix $\calA_p$ weights in terms of reducing operators and averaging operators.   We first consider reducing operators.
Reducing operators play an important role in the theory of matrix weights; for details and references, see~\cite{bownik_extrapolation_2022}.  Here, we recall the definition of reducing operators in the variable exponent setting given in~\cite{ConvOpsOnVLS}.  For this definition we need the following result.
\begin{theorem}\cite[Theorem 4.11]{bownik_extrapolation_2022}\label{thm:EllipsoidApprox}
Given a measurable norm function $r: \R^n\times \R^d \to [0,\infty)$, there exists a positive-definite, measurable matrix function $W:\R^n\to \calS_d$ such that for any $x\in \R^n$ and $\ve\in \R^d$, 
\[r(x,\ve) \leq |W(x) \ve| \leq \sqrt{d} r(x,\ve).\]
\end{theorem}
Given a matrix weight $W:\R^n \to \calS_d$, define the norm function $r(\cdot,\cdot): \R^n\times \R^d \to [0,\infty)$ by $r(x,\ve) = |W(x) \ve|$. Given a cube $Q \subset \R^n$, define the norm $\langle r\rangle_{\pp,Q}: \R^d \to [0,\infty)$ by 
\[\langle r\rangle_{\pp,Q} (\ve) : = |Q|^{-\frac{1}{p_Q}} \| r(\cdot, \ve) \chi_Q(\cdot)\|_{\Lpp(\R^n)} = |Q|^{-\frac{1}{p_Q}} \| \ve\chi_Q\|_{\Lpp(W)}.\]
By Theorem \ref{thm:EllipsoidApprox}, there exists a positive-definite, self-adjoint, (constant) matrix $\calW_Q^\pp$ such that $\langle r\rangle_{\pp,Q} (\ve) \approx |\calW_Q^\pp \ve|$ for all $\ve\in\R^d$. We call $\calW_Q^\pp$ the reducing operator associated to $W$ on $Q$. 

Now let $r^*(x,\cdot)$ be the dual norm of $r(x,\cdot)$, given by $r^*(x,\ve) = |W^{-1}(x) \ve|$. Define the norm $\langle r^* \rangle_{\cpp,Q} :\R^d \to [0,\infty)$ by 
\[\langle r^* \rangle_{\cpp,Q}(\ve) : = |Q|^{-1/p'_Q} \| r^*(\cdot, \ve) \chi_Q(\cdot)\|_{\Lcpp(\R^n)} = |Q|^{-1/p'_Q} \| \ve\chi_Q\|_{\Lcpp(W^{-1})}.\]

Again by Theorem \ref{thm:EllipsoidApprox}, there exists a positive-definite, self-adjoint, (constant) matrix $\overline{\calW}_Q^\cpp$ such that $\langle r^* \rangle_{\cpp,Q}(\ve) \approx |\overline{\calW}_Q^{\cpp}\ve|$ for all $\ve\in \R^d$. We call $\overline{\calW}_Q^\cpp$ the reducing operator associated to $W^{-1}$ on $Q$. 

We can use these two reducing operators to characterize $\calA_\pp$.

\begin{prop}\cite[Proposition 4.7]{ConvOpsOnVLS}\label{thm:ReducOpEquivAvgOp}
Let $\pp \in \Pp(\R^n)$ and $W : \R^n\to \calS_d$ be a matrix weight. Then $W \in \calA_\pp$ if and only if 
\[[W]_{\calA_\pp}^R : = \sup_Q |\calW_Q^\pp \overline{\calW}_Q^\cpp |_{\op} <\infty.\]
Moreover, $[W]_{\calA_\pp}^R \approx [W]_{\calA_\pp}$ with implicit constants depending only on $d$.
\end{prop}

We can also characterize $\calA_\pp$ weights using averaging operators.  In the constant exponent case, this idea first appeared in Goldberg~\cite{goldberg_matrix_2003} and was extensively developed in~\cite{cruz-uribe_matrix_2016}.  Given a cube $Q$ and a matrix weight $W: \R^n\to \calS_d$, define the averaging operator $A_{W,Q}$ by 
\[ A_{W,Q} \vf(x) : = \dashint_Q W(x) W^{-1}(y) \vf(y) \, dy \, \chi_Q(x),\]
for any vector-valued function $\vf : \R^n\to \R^d$. 
\begin{prop}\cite[Theorem 4.1]{ConvOpsOnVLS}\label{thm:calAppEquivAvgOp}
Let $\pp \in \Pp(\R^n)$. A matrix weight $W : \R^n\to \calS_d$ is a matrix $\calA_\pp$ weight if and only if $A_{W,Q}: \Lpp(\R^n;\R^d) \to \Lpp(\R^n;\R^d)$. Moreover,
\[ \frac{1}{K_\pp} \sup_Q \| A_{W,Q}\|\leq   [W]_{\calA_\pp} \leq 4C(d) \sup_Q \| A_{W,Q}\|,\]
where $\|A_{W,Q}\|$ is the operator norm of $A_{W,Q}$ from $\Lpp(\R^n;\R^d)$ to $\Lpp(\R^n;\R^d)$. 
\end{prop}
\begin{remark}\label{rem:Theorem4.1Weighted}
By linearity, $A_{W,Q}: \Lpp(\R^n;\R^d)\to \Lpp(\R^n;\R^d)$ uniformly over all cubes $Q$ if and only if $A_Q : \Lpp(W)\to \Lpp(W)$ uniformly over all cubes and  
\[\|A_{W,Q}\|_{\Lpp(\R^n;\R^d)\to \Lpp(\R^n;\R^d)} = \| A_Q\|_{\Lpp(W)\to \Lpp(W)},\]
where $A_Q\vf= \dashint_Q \vf(y)\, dy \, \chi_Q$.
\end{remark}

The following proposition relates matrix $\calA_\pp$ weights to scalar $\calA_\pp$ weights.  The main idea of the proof is contained in~\cite[Proposition~4.8]{ConvOpsOnVLS}; we give the details here as we need to keep careful track of the constants.
\begin{lemma}\label{prop:MatrixToScalarCalApp}
Let $\pp \in \Pp(\R^n)$ and $W: \R^n\to \calS_d$ be a matrix weight.  If $W \in \calA_\pp$, then for all nonzero $\ve\in \R^d$, $|W(\cdot)\ve|$ is a scalar $\calA_\pp$ weight with 
\[[|W(\cdot)\ve|]_{\calA_\pp} \leq 4K_\pp [W]_{\calA_\pp}.\]
\end{lemma}
\begin{proof}
Let $\pp \in \Pp(\R^n)$, $W \in \calA_\pp$, and fix $\ve \in \R^d$ with $\ve\neq 0$. Let $w=|W\ve|$.  Fix a cube $Q$.  We will show that $A_Q : L^\pp(w)\rightarrow L^\pp(w)$ and that 
\begin{equation} \label{eqn:scalarAp}
 \|A_Q\|_{L^\pp(w)\rightarrow L^\pp(w)} \leq K_\pp [W]_{\calA_\pp}. 
 \end{equation}
Then by Proposition~\ref{thm:calAppEquivAvgOp} and Remark~\ref{rem:Theorem4.1Weighted}, which both hold in the scalar case (i.e., when $d=1$), we will have that 
\[ [w]_{\calA_\pp} \leq 4 \sup_Q \|A_Q\|_{L^\pp(w)\rightarrow L^\pp(w)} \leq 4K_\pp [W]_{\calA_\pp}, \]
which is what we want to prove.

Let $\phi$ be any scalar function and define $\vf=\phi \ve$. Fix a cube $Q$. Observe that
\begin{multline*}
\| A_Q\vf\|_{\Lpp(W)}  = \big\| |W(\cdot) A_Q\vf(\cdot)|\big\|_{\Lpp(\R^n)}
	= \bigg\| \bigg| W(\cdot) \dashint_Q \phi(y)\ve\, dy \bigg|\chi_Q(\cdot)\bigg\|_{\Lpp(\R^n)}\\
	 = \bigg\| \bigg| W(\cdot) \ve \dashint_Q \phi(y)\, dy\bigg|\chi_Q(\cdot) \bigg\|_{\Lpp(\R^n)}
	 = \bigg\| |W(\cdot)\ve|\bigg| \dashint_Q \phi(y)\, dy \bigg| \chi_Q(\cdot)\bigg\|_{\Lpp(\R^n)}
	 = \| A_Q \phi\|_{\Lpp(w)}.
\end{multline*}
Similarly,
\begin{multline*}
\| \vf\|_{\Lpp(W)}  = \| |W(\cdot)\vf|\|_{\Lpp(\R^n)}\\
	 = \| |W(\cdot) \phi(\cdot)\ve|\|_{\Lpp(\R^n)}
	 = \| |W(\cdot)\ve| \phi(\cdot)\|_{\Lpp(\R^n)}
	 = \| \phi\|_{\Lpp(w)}.
\end{multline*}
Then, since $W \in \calA_\pp$, by Proposition~\ref{thm:calAppEquivAvgOp} and Remark \ref{rem:Theorem4.1Weighted}, we have that
\begin{multline*}
\|A_Q \phi\|_{\Lpp(w)} = \| A_Q \vf\|_{\Lpp(W)}
	 \leq \|A_Q\|_{\Lpp(W)\to \Lpp(W)}\| \vf\|_{\Lpp(W)} \\
	 \leq K_\pp [W]_{\calA_\pp} \| \vf\|_{\Lpp(W)}
	 = K_\pp [W]_{\calA_\pp}\| \phi\|_{\Lpp(w)}.
\end{multline*}
Hence, \eqref{eqn:scalarAp} and our proof is complete.
\end{proof}

The following property of reducing operators is  used in the proof of Theorem \ref{thm:calAppRightOpen}.
\begin{lemma}\label{lem:NormAvgUnifBound1}
Let $\pp \in \Pp(\R^n)\cap LH(\R^n)$ with $p_+<\infty$, and let $W: \R^n\to \calS_d$ be a matrix weight. If $W \in \calA_\pp$, then there exists $r>1$ such that for all $s \in [1,r]$, 
\[  \| |W(\cdot) (\calW_Q^\pp)^{-1} |_{\op} \chi_Q(\cdot)\|_{L^{s\pp}(\R^n)} \lesssim \|\chi_Q\|_{L^{s\pp}(\R^n)},\]
for all cubes $Q \subset \R^n$. The implicit constant depends on $d, \pp, C_\infty, C^*$ and $[W]_{\calA_\pp}$.
\end{lemma}
\begin{proof}
Let $\pp\in \Pp(\R^n) \cap LH(\R^n)$ and $W \in \calA_\pp$. Let $\{\ve_i\}_{i=1}^d$ be the coordinate basis of $\R^d$. Fix $Q$. For each $i=1,\ldots, d$, consider the scalar weight $w_i^Q(\cdot) = |W(\cdot) (\calW_Q^\pp)^{-1} \ve_i|$. For each of these weights, choose $r_i^Q$ to be the reverse H\"{o}lder exponent and $C_i^Q$ to be the constant from Theorem \ref{thm:NormRH}, i.e., 
\[ r_i^Q = 1 + \frac{1}{C_* \left([|W(\cdot) (\calW_Q^\pp)^{-1}\ve_i|]_{\calA_\pp}^{1+2\frac{C_\infty p_+}{p_\infty p_-}}\right)^{p_+}},\]
and
\[ C_i^Q = C^* [|W(\cdot)(\calW_Q^\pp)^{-1}\ve_i|]_{\calA_\pp}^{\frac{C_\infty p_+}{{p_-}^2}(p_+ +1)}.\]
By Lemma \ref{prop:MatrixToScalarCalApp}, $[W(\cdot) (\calW_Q^\pp)^{-1}\ve_i|]_{\calA_\pp} \leq C(d) K_\pp [W]_{\calA_\pp}$. Thus, defining 
\[ r= 1+ \frac{1}{C_* \left((4 K_\pp [W]_{\calA_\pp})^{1+2\frac{C_\infty p_+}{p_\infty p_-}}\right)^{p_+}},\]
and 
\[ M_\pp = C^* (4 K_\pp[W]_{\calA_\pp})^{\frac{C_\infty p_+}{p_-^2}(p_++1)},\]
we have $r \leq r_i^Q$ and $M_\pp \geq C_i^Q$ for each $i$. Then by Lemma \ref{opNorm:equiv}, the triangle inequality, Corollary \ref{cor:NormRHCollapsing}, and the definition of the reducing operator $\calW_Q^\pp$, we get for all $s\in [1,r]$
\begin{align*}
 & |Q|^{-\frac{1}{sp_Q}} \| |W(\cdot)(\calW_Q^\pp)^{-1}|_{\op} \chi_Q(\cdot)\|_{L^{s\pp}(\R^n)} \\
    &\qquad \leq |Q|^{-\frac{1}{sp_Q}} \left\| \sum_{i=1}^d |W(\cdot) (\calW_Q^\pp)^{-1} \ve_i|\chi_Q(\cdot)\right\|_{L^{s\pp}(\R^n)}\\
    & \qquad\leq \sum_{i=1}^d |Q|^{-\frac{1}{sp_Q}} \| |W(\cdot) (\calW_Q^\pp)^{-1} \ve_i|\chi_Q(\cdot)\|_{L^{s\pp}(\R^n)}\\
    & \qquad\leq 32[1]_{\calA_\pp} \sum_{i=1}^d C_i^Q |Q|^{-\frac{1}{p_Q}} \| |W(\cdot) (\calW_Q^\pp)^{-1} \ve_i|\chi_Q(\cdot)\|_{\Lpp(\R^n)}\\
    & \qquad\leq 32 [1]_{\calA_\pp} M_\pp \sum_{i=1}^d |Q|^{-\frac{1}{p_Q}} \| |W(\cdot) (\calW_Q^\pp)^{-1} \ve_i|\chi_Q(\cdot)\|_{\Lpp(\R^n)}\\
    & \qquad\leq 32 [1]_{\calA_\pp} M_\pp \sum_{i=1}^d |\calW_Q^\pp (\calW_Q^\pp)^{-1} \ve_i|\\
    & \qquad\leq 32 [1]_{\calA_\pp} M_\pp d.
\end{align*}

Multiplying both sides by $|Q|^{\frac{1}{sp_Q}}$ and applying Lemma \ref{CharFunctionNormIneq}, we get
\[ \| |W(\cdot)(\calW_Q^\pp)^{-1}|_{\op} \chi_Q(\cdot)\|_{L^{s\pp}(\R^n)} \leq 32 [1]_{\calA_\pp} M_\pp d 2K_{s\pp} \|\chi_Q\|_{L^{s\pp}(\R^n)}.\]
Since $p_+<\infty$, $K_{s\pp} \leq 3$, so the final constant depends only on $d, \pp, C_\infty, C^*$, and $[W]_{\calA_\pp}$.
\end{proof}

\medskip

Finally, in our proofs Theorems \ref{thm:calAppRightOpen} and \ref{thm:calAppLeftOpen}, we will need to use an auxiliary averaging operator defined using a reducing operator. Given $\pp \in \Pp(\R^n)$, define $\tA_{W,\pp,Q}$ by 
\[ \tA_{W,\pp,Q} \vf(x) : = \dashint_Q \calW_Q^\pp W^{-1}(y) \vf(y) \, dy \, \chi_Q(x).\]
This operator plays a role analogous to the auxiliary maximal operator introduced by Goldberg~\cite{goldberg_matrix_2003}; we introduce it because we cannot assume that the Goldberg auxiliary maximal operator is bounded on $L^\pp$.  We will consider this problem in a forthcoming paper~\cite{dcu-mp-aux-max}.  We will prove that $\tA_{W,\pp,Q}$ is bounded on $L^\pp$ and that it satisfies a right-openness property. 

\begin{lemma}\label{lem:AuxAvgOpBounded}
Let $\pp \in \Pp(\R^n)\cap LH(\R^n)$ with $p_+ <\infty$ and $W : \R^n\to \calS_d$ be a matrix weight. If $W \in \calA_\pp$, then there exists $r>1$ such that for all $s\in [1,r]$, 
\[ \tA_{W,\pp,Q} :L^{s\pp}(\R^n;\R^d) \to L^{s\pp}(\R^n;\R^d) \]
uniformly over all cubes $Q$.
\end{lemma}

\begin{proof}
Let $\pp \in \Pp(\R^n)\cap LH(\R^n)$ and $W\in \calA_\pp$. As in the proof of Lemma \ref{lem:NormAvgUnifBound1}, define
\[ r= 1 + \frac{1}{C_* \left((4 K_\pp [W]_{\calA_\pp})^{1+2\frac{C_\infty p_+}{p_\infty p_-}}\right)^{p_+}}.\]

Let $s \in [1,r]$. Fix $\vf\in L^{s\pp}(\R^n;\R^d)$. Observe that by homogeneity of the $L^{s\pp}$ norm, the triangle inequality, H\"{o}lder's inequality, and Lemma \ref{CharFunctionNormIneq}, we have
\begin{align*}
 & \| \tA_{W,\pp,Q}\vf\|_{L^{s\pp}(\R^n;\R^d)} \\
    & \qquad = \left\| \left| \dashint_Q \calW_Q^\pp W^{-1}(y) \vf(y)\, dy \right| \chi_Q(\cdot)\right\|_{L^{s\pp}(\R^n)}\\
    & \qquad = |Q|^{-1} \left| \int_Q \calW_Q^\pp W^{-1}(y) \vf(y)\, dy \right| \| \chi_Q \|_{L^{s\pp}(\R^n)}\\
    & \qquad \leq |Q|^{-1} \int_Q |\calW_Q^\pp W^{-1}|_{\op} |\vf(y)|\, dy  \| \chi_Q\|_{L^{s\pp}(\R^n)}\\
    & \qquad \leq K_{s\pp} |Q|^{-1} \||\calW_Q^\pp W^{-1}(\cdot)|_{\op} \chi_Q)(\cdot) \|_{L^{(s\pp)'}(\R^n)} \| \vf\|_{L^{s\pp}(\R^n;\R^d)} \| \chi_Q\|_{L^{s\pp}(\R^n)}\\
    & \qquad \leq 4 K_{s\pp}^3 [1]_{\calA_{s\pp}}|Q|^{-1} \| |\calW_Q^\pp W^{-1}(\cdot)|_{\op} \chi_Q(\cdot)\|_{L^{(s\pp)'}(\R^n)} \| \vf\|_{L^{s\pp}(\R^n;\R^d)} |Q|^{\frac{1}{sp_Q}}\\
    & \qquad = 4K_{s\pp}^3 [1]_{\calA_{s\pp}} |Q|^{-\frac{1}{(sp_Q)'}} \| |\calW_Q^\pp W^{-1}(\cdot)|_{\op} \chi_Q(\cdot)\|_{L^{(s\pp)'}(\R^n)} \| \vf\|_{L^{s\pp}(\R^n;\R^d)}.
\end{align*}

To finish the proof, we must show that $|Q|^{-\frac{1}{(sp_Q)'}} \| |\calW_Q^\pp W^{-1}(\cdot)|_{\op} \chi_Q(\cdot)\|_{L^{(s\pp)'}(\R^n)}$ is bounded by a constant independent of $Q$. Define $\vp$ by 
\[ \frac{1}{(sp(x))'} = \frac{1}{p'(x)} + \frac{1}{v(x)}.\]
Then by the generalized H\"{o}lder's inequality on variable Lebesgue spaces, Lemma \ref{lem:GeneralizedHolder}, and Lemma~\ref{CharFunctionNormIneq},
\begin{align*} \| |\calW_Q^\pp W^{-1}(\cdot) |_{\op} \chi_Q(\cdot) \|_{L^{(s\pp)'}(\R^n)} &  \leq  K \| |\calW_Q^\pp W^{-1}(\cdot)|_{\op} \chi_Q(\cdot) \|_{\Lcpp(\R^n)} \| \chi_Q\|_{L^{\vp}(\R^n)}\\
    & \leq K 4K_{\vp}^2 [1]_{\calA_{\vp}} \| |\calW_Q^\pp W^{-1}(\cdot) |_{\op} \chi_Q(\cdot) \|_{\Lcpp(\R^n)} |Q|^{\frac{1}{v_Q}}.
\end{align*}

Combining this with Lemma \ref{SelfAdjointCommutes}, Lemma \ref{opNorm:equiv}, the triangle inequality, the definition of the reducing operator $\overline{\calW}_Q^\cpp$, and Proposition \ref{thm:ReducOpEquivAvgOp}, we get
\begin{align*}
& |Q|^{-\frac{1}{(sp_Q)'}} \| |\calW_Q^\pp W^{-1}(\cdot)|_{\op} \chi_Q(\cdot) \|_{L^{(s\pp)'}(\R^n)} \\
    & \qquad \leq 4K K_{\vp}^2  [1]_{\calA_\vp} |Q|^{-\frac{1}{(sp_Q)'}} \| |\calW_Q^\pp W^{-1}(\cdot)|_{\op} \chi_Q(\cdot) \|_{\Lcpp(\R^n)} |Q|^{\frac{1}{v_Q}}\\
    & \qquad = 4 K K_\vp^2 [1]_{\calA_\vp} |Q|^{-\frac{1}{p'_Q}} \| |\calW_Q^\pp W^{-1}(\cdot)|_{\op} \chi_Q(\cdot)\|_{\Lcpp(\R^n)}\\
    & \qquad = 4K K_\vp^2 [1]_{\calA_\vp} |Q|^{-\frac{1}{p'_Q}} \| |W^{-1}(\cdot)\calW_Q^\pp|_{\op} \chi_Q(\cdot)\|_{\Lcpp(\R^n)}\\
    & \qquad \leq 4 K K_\vp^2 [1]_{\calA_\vp} \sum_{i=1}^d |Q|^{-\frac{1}{p'_Q}} \| |W^{-1}(\cdot)\calW_Q^\pp \ve_i|\chi_Q(\cdot)\|_{\Lcpp(\R^n)}\\
    & \qquad \leq 4 K K_\vp^2 [1]_{\calA_\vp} \sum_{i=1}^d | \overline{\calW}_Q^\cpp \calW_Q^\pp \ve_i|\\
    & \qquad \leq 4K K_\vp^2 [1]_{\calA_\vp} d |\overline{\calW}_Q^\cpp \calW_Q^\pp |_{\op} \\
    & \qquad \leq 4 K K_\vp^2 [1]_{\calA_\vp} d [W]_{\calA_\pp}^R.
\end{align*}

Since $W \in \calA_\pp$, $[W]_{\calA_\pp}^R<\infty$ by Lemma \ref{thm:ReducOpEquivAvgOp}. Since the final constant is independent of $Q$, $\tA_{W,\pp, Q} : L^{s\pp}(\R^n;\R^d)\to L^{s\pp}(\R^n;\R^d)$ uniformly over all cubes $Q$. This completes the proof.

\end{proof}
%

%
\section{Right and Left-openness of the $\calA_\pp$ Classes}\label{sec:LeftRightOpen}
In this section we prove Theorems \ref{thm:calAppRightOpen} and~\ref{thm:calAppLeftOpen}.
\begin{proof}[Proof of Theorem \ref{thm:calAppRightOpen}]
Let $\pp \in \Pp(\R^n) \cap LH(\R^n)$ with $p_+<\infty$ and $W: \R^n\to \calS_d$ be a matrix $\calA_\pp$ weight. Choose $r$ from Lemma \ref{lem:NormAvgUnifBound1}, and let $s\in [1,r]$. By Proposition \ref{thm:calAppEquivAvgOp}, to prove $W \in \calA_{s\pp}$, it suffices to show $A_{W,Q} : L^{s\pp}(\R^n;\R^d) \to L^{s\pp}(\R^n;\R^d)$ uniformly over all cubes $Q$.

Fix a cube $Q$ and $\vf \in L^{s\pp}(\R^n;\R^d)$. Observe that by homogeneity of the $L^{s\pp}$ norm and Lemma \ref{lem:NormAvgUnifBound1},
\begin{align*}
& \| A_{W,Q} \vf\|_{L^{s\pp}(\R^n;\R^d)} \\
    & \qquad = \left\| \left| \dashint_Q W(\cdot) W^{-1}(y) \vf(y) \, dy \right| \chi_Q(\cdot)\right\|_{L^{s\pp}(\R^n)}\\
    & \qquad = \left\| \left| W(\cdot) (\calW_Q^\pp)^{-1} \dashint_Q \calW_Q^\pp W^{-1}(y) \vf(y)\, dy \right| \chi_Q(\cdot)\right\|_{L^{s\pp}(R^n)}\\
    & \qquad \leq \| |W(\cdot) (\calW_Q^\pp)^{-1}|_{\op}\chi_Q(\cdot)\|_{L^{s\pp}(\R^n)} \left| \dashint_Q \calW_Q^\pp W^{-1}(y) \vf(y)\, dy \right|\\
    & \qquad \lesssim \|\chi_Q\|_{L^{s\pp}(\R^n)} \left| \dashint_Q \calW_Q^\pp W^{-1}(y) \vf(y)\, dy \right|\\
    & \qquad = \left\| \left| \dashint_Q \calW_Q^\pp W^{-1}(y) \vf(y) \, dy \right| \chi_Q(\cdot) \right\|_{L^{s\pp}(\R^n)}\\
    & \qquad = \| \tA_{W,\pp,Q} \vf\|_{L^{s\pp}(\R^n)}.
\end{align*}

Since the choice of $r$ is the same as in Lemma \ref{lem:AuxAvgOpBounded}, $\tA_{W,\pp,Q}$ is bounded on $L^{s\pp}(\R^n;\R^d)$ uniformly over all cubes $Q$. Thus, we have shown
\[ \| A_{W,Q}\vf\|_{L^{s\pp}(\R^n;\R^d)} \lesssim \| \vf\|_{L^{s\pp}(\R^n;\R^d)}, \]
with implicit constant independent of $Q$. Hence, $W \in \calA_{s\pp}$. 
\end{proof}

We now prove Theorem \ref{thm:calAppLeftOpen}. The proof is almost identical to the proof of Theorem \ref{thm:calAppRightOpen}, so we just sketch the key steps. 
\begin{proof}[Proof of Theorem \ref{thm:calAppLeftOpen}]
Fix $\pp\in \Pp(\R^n)\cap LH(\R^n)$ with $p_- >1$ and $W \in \calA_\pp$. Choose $r$ from Lemma \ref{lem:NormAvgUnifBound1} and let $s\in [1,r]$. Define $\cqp = s\cpp$. By Remark \ref{rem:calAppSymmetric}, to prove $W \in \calA_\qp$, it suffices to show that $W^{-1}\in \calA_\cqp$. By Lemma \ref{thm:calAppEquivAvgOp}, we prove $W^{-1}\in \calA_\cqp$ by showing $A_{W^{-1},Q}: \Lcqp(\R^n;\R^d) \to \Lcqp(\R^n;\R^d)$ uniformly over all cubes. 

Fix a cube $Q$ and $\vf\in \Lcqp(\R^n;\R^d)$. We repeat the same steps as in the proof of Theorem~\ref{thm:calAppRightOpen}; since $p_->1$, we have $(p')_+<\infty$, so we can apply Lemma~\ref{lem:NormAvgUnifBound1} to $W^{-1}$ and $\cpp$ to get
\[ \| A_{W^{-1},Q} \vf\|_{\Lcqp(\R^n;\R^d)} \lesssim \| \tA_{W^{-1}, \cpp, Q} \vf\|_{\Lcqp(\R^n;\R^d)}.\]
By our choice of $r$, we may apply Lemma \ref{lem:AuxAvgOpBounded} to $\tA_{W^{-1},\cpp, Q}$ to get 
\[ \| \tA_{W^{-1},\cpp, Q} \vf\|_{\Lcqp(\R^n;\R^d)} \lesssim \| \vf\|_{\Lcqp(\R^n;\R^d)},\]
with the implicit constant independent of $Q$. Thus, we have shown
\[ A_{W^{-1},Q} : \Lcqp(\R^n;\R^d)\to \Lcqp(\R^n;\R^d),\]
uniformly over all cubes $Q$. Hence, $W^{-1}\in \calA_{\cqp}$, and so $W\in \calA_{\qp}$.
\end{proof}

\bibliography{Bibliography.bib}{}

\begin{thebibliography}{10}

\bibitem{MR1857041}
M.~Bownik.
\newblock Inverse volume inequalities for matrix weights.
\newblock {\em Indiana Univ. Math. J.}, 50(1):383--410, 2001.

\bibitem{bownik_extrapolation_2022}
M.~Bownik and D.~Cruz-Uribe.
\newblock Extrapolation and factorization of matrix weights.
\newblock {\em preprint}, 2023.
\newblock arXiv:2210.09443.

\bibitem{coifman_weighted_1974}
R.~Coifman and C.~Fefferman.
\newblock Weighted norm inequalities for maximal functions and singular integrals.
\newblock {\em Studia Math.}, 51:241--250, 1974.

\bibitem{cruzuribe2017extrapolation}
D.~Cruz-Uribe.
\newblock Extrapolation and factorization.
\newblock In J.~Lukes and L.~Pick, editors, {\em Function spaces, embeddings and extrapolation X, Paseky 2017}, pages 45--92. Matfyzpress, Charles University, 2017.
\newblock arXiv:1706.02620.

\bibitem{MR4387458}
D.~Cruz-Uribe and J.~Cummings.
\newblock Weighted norm inequalities for the maximal operator on {$L^{p(\cdot)}$} over spaces of homogeneous type.
\newblock {\em Ann. Fenn. Math.}, 47(1):457--488, 2022.

\bibitem{MR2837636}
D.~Cruz-Uribe, L.~Diening, and P.~H\"{a}st\"{o}.
\newblock The maximal operator on weighted variable {L}ebesgue spaces.
\newblock {\em Fract. Calc. Appl. Anal.}, 14(3):361--374, 2011.

\bibitem{cruz-uribe_variable_2013}
D.~Cruz-Uribe and A.~Fiorenza.
\newblock {\em Variable {L}ebesgue spaces}.
\newblock Applied and Numerical Harmonic Analysis. Birkh\"{a}user/Springer, Heidelberg, 2013.
\newblock Foundations and harmonic analysis.

\bibitem{MR2927495}
D.~Cruz-Uribe, A.~Fiorenza, and C.~Neugebauer.
\newblock Weighted norm inequalities for the maximal operator on variable {L}ebesgue spaces.
\newblock {\em J. Math. Anal. Appl.}, 394(2):744--760, 2012.

\bibitem{cruz-uribe_matrix_2016}
D.~Cruz-Uribe, K.~Moen, and S.~Rodney.
\newblock Matrix {$\mathcal{A}_p$} weights, degenerate {S}obolev spaces, and mappings of finite distortion.
\newblock {\em J. Geom. Anal.}, 26(4):2797--2830, 2016.

\bibitem{dcu-mp-aux-max}
D.~Cruz-Uribe and M.~Penrod.
\newblock The boundedness of the auxiliary maximal operator on {$L^\pp$}.
\newblock {\em In preparation}.

\bibitem{ConvOpsOnVLS}
D.~Cruz-Uribe and M.~Penrod.
\newblock Convolution operators in matrix weighted, variable lebesgue spaces.
\newblock {\em Analysis and Applications}, 22(07):1133--1157, 2024.

\bibitem{TroyThesis}
D.~Cruz-Uribe and T.~Roberts.
\newblock Necessary conditions for the boundedness of fractional operators on variable {L}ebesgue spaces.
\newblock {\em preprint}, 2024.
\newblock arXiv2408.12745.

\bibitem{MR3572271}
D.~Cruz-Uribe and L.~Wang.
\newblock Extrapolation and weighted norm inequalities in the variable {L}ebesgue spaces.
\newblock {\em Trans. Amer. Math. Soc.}, 369(2):1205--1235, 2017.

\bibitem{diening_lebesgue_2011}
L.~Diening, P.~Harjulehto, P.~H\"{a}st\"{o}, and M.~R\r{u}\v{z}i\v{c}ka.
\newblock {\em Lebesgue and {S}obolev spaces with variable exponents}, volume 2017 of {\em Lecture Notes in Mathematics}.
\newblock Springer, Heidelberg, 2011.

\bibitem{duoandikoetxea_fourier_2000}
J.~Duoandikoetxea.
\newblock {\em Fourier analysis}, volume~29 of {\em Graduate Studies in Mathematics}.
\newblock American Mathematical Society, Providence, RI, 2001.
\newblock Translated and revised from the 1995 Spanish original by David Cruz-Uribe.

\bibitem{MR3473651}
J.~Duoandikoetxea, F.~Mart\'{\i}n-Reyes, and S.~Ombrosi.
\newblock On the {$A_\infty$} conditions for general bases.
\newblock {\em Math. Z.}, 282(3-4):955--972, 2016.

\bibitem{MR0736522}
E.~Dyn{$'$}kin and B.~Osilenker.
\newblock Weighted estimates for singular integrals and their applications.
\newblock In {\em Mathematical analysis, {V}ol. 21}, Itogi Nauki i Tekhniki, pages 42--129. Akad. Nauk SSSR, Vsesoyuz. Inst. Nauchn. i Tekhn. Inform., Moscow, 1983.

\bibitem{MR0807149}
J.~Garc\'{\i}a-Cuerva and J.~Rubio~de Francia.
\newblock {\em Weighted norm inequalities and related topics}, volume 116 of {\em North-Holland Mathematics Studies}.
\newblock North-Holland Publishing Co., Amsterdam, 1985.
\newblock Notas de Matem\'{a}tica, 104. [Mathematical Notes].

\bibitem{goldberg_matrix_2003}
M.~Goldberg.
\newblock Matrix {$A_p$} weights via maximal functions.
\newblock {\em Pacific J. Math.}, 211(2):201--220, 2003.

\bibitem{MR3092729}
T.~Hyt\"{o}nen and C.~P\'{e}rez.
\newblock Sharp weighted bounds involving {$A_\infty$}.
\newblock {\em Anal. PDE}, 6(4):777--818, 2013.

\bibitem{MR2990061}
T.~Hyt\"onen, C.~P\'erez, and E.~Rela.
\newblock Sharp reverse {H}\"older property for {$A_\infty$} weights on spaces of homogeneous type.
\newblock {\em J. Funct. Anal.}, 263(12):3883--3899, 2012.

\bibitem{MR3682615}
A.~K. Lerner.
\newblock On a dual property of the maximal operator on weighted variable {$L^p$} spaces.
\newblock In {\em Functional analysis, harmonic analysis, and image processing: a collection of papers in honor of {B}j\"orn {J}awerth}, volume 693 of {\em Contemp. Math.}, pages 283--300. Amer. Math. Soc., Providence, RI, 2017.

\bibitem{MR0293384}
B.~Muckenhoupt.
\newblock Weighted norm inequalities for the {H}ardy maximal function.
\newblock {\em Trans. Amer. Math. Soc.}, 165:207--226, 1972.

\bibitem{MR340523}
B.~Muckenhoupt and R.~Wheeden.
\newblock Weighted norm inequalities for fractional integrals.
\newblock {\em Trans. Amer. Math. Soc.}, 192:261--274, 1974.

\bibitem{nazarov_hunt_1996}
F.~Nazarov and S.~Treil.
\newblock The hunt for a {Bellman} function: applications to estimates of singular integral operators and to other classical problems in harmonic analysis.
\newblock {\em Algebra i Analiz}, 8(5):32--162, 1996.

\bibitem{roudenko_matrix-weighted_2002}
S.~Roudenko.
\newblock Matrix-weighted {B}esov spaces.
\newblock {\em Trans. Amer. Math. Soc.}, 355(1):273--314, 2003.

\bibitem{treil_wavelets_1997}
S.~Treil and A.~Volberg.
\newblock Wavelets and the angle between past and future.
\newblock {\em J. Funct. Anal.}, 143(2):269--308, 1997.

\end{thebibliography}
\bibliographystyle{plain}

\end{document}